\documentclass{amsart}
\usepackage{amsthm,amssymb,mathtools}
\usepackage[british]{babel}
\usepackage{enumitem}
\usepackage[latin1]{inputenc}
\usepackage{bussproofs}
\usepackage{url}
\usepackage{hyperref}
\usepackage[hang,flushmargin]{footmisc}

\newtheorem{theorem}{Theorem}
\numberwithin{theorem}{section}
\newtheorem{corollary}[theorem]{Corollary}
\newtheorem{lemma}[theorem]{Lemma}
\newtheorem{proposition}[theorem]{Proposition}
\theoremstyle{definition}
\newtheorem{definition}[theorem]{Definition}

\newenvironment{bprooftree}
  {\leavevmode\hbox\bgroup}
  {\DisplayProof\egroup}

\newcommand{\con}{\operatorname{Con}}
\newcommand{\pro}{\operatorname{Proof}}
\newcommand{\true}{\operatorname{Tr}}
\newcommand{\pr}{\operatorname{Pr}}

\newcommand{\dcon}{\operatorname{Con}^\diamond}
\newcommand{\pa}{\mathbf{PA}}
\newcommand{\aca}{\mathbf{ACA}_0}
\newcommand{\ea}{\mathbf{EA}}
\newcommand{\feps}{F_{\varepsilon_0}}
\newcommand{\isigma}{\mathbf{I\Sigma}}
\newcommand{\drfn}{\operatorname{RFN}_{\mathbf{PA}}^\diamond}
\newcommand{\prfn}{{\operatorname{Rfn}}}
\newcommand{\rfn}{\operatorname{RFN}}
\newcommand{\ti}{\operatorname{TI}}
\newcommand{\prog}{\operatorname{Prog}}
\newcommand{\tis}{\operatorname{SeqTI}}
\newcommand{\progs}{\operatorname{SeqProg}}
\newcommand{\len}{\operatorname{len}}

\makeatletter
\newcommand{\dotminus}{\mathbin{\text{\@dotminus}}}
\newcommand{\@dotminus}{%
  \ooalign{\hidewidth\raise1ex\hbox{.}\hidewidth\cr$\m@th-$\cr}%
}
\makeatother

\title[Slow reflection]{Slow reflection\footnotemark[1]}
\author{Anton Freund}

\begin{document}

\begin{abstract}
We describe a ``slow'' version of the hierarchy of uniform reflection principles over Peano Arithmetic ($\pa$). These principles are unprovable in Peano Arithmetic (even when extended by usual reflection principles of lower complexity) and introduce a new provably total function. At the same time the consistency of $\pa$ plus slow reflection is provable in $\pa+\con(\pa)$. We deduce a conjecture of \mbox{S.-D.}~Friedman, Rathjen and Weiermann: Transfinite iterations of slow consistency generate a hierarchy of precisely $\varepsilon_0$ stages between $\pa$ and $\pa+\con(\pa)$ (where $\con(\pa)$ refers to the usual consistency statement).
\end{abstract}

\maketitle
{\let\thefootnote\relax\footnotetext{\copyright~2017. This manuscript version is made available under the CC-BY-NC-ND 4.0 license \url{http://creativecommons.org/licenses/by-nc-nd/4.0/}. The paper has been accepted for publication in the Annals of Pure and Applied Logic (doi 10.1016/j.apal.2017.06.003).}}

The starting point for our work is the notion of slow consistency for (finite extensions of) Peano Arithmetic that has been introduced by Sy-David Friedman, Michael Rathjen and Andreas Weiermann in \cite{rathjen13}. Up to an ``index shift'' (see below) it is defined as
\begin{equation}\label{eq:def-slow-consistency}
 \dcon(\pa+\varphi):\equiv\forall_x(\feps(x)\!\downarrow\,\rightarrow\con(\isigma_{x+1}+\varphi)).
\end{equation}
This formula involves the function $\feps$ at stage $\varepsilon_0$ of the fast-growing hierarchy, due to Wainer and Schwichtenberg \cite{wainer70,schwichtenberg71}. We work with the version used by Sommer \cite{sommer95}: Adopting his assignment of fundamental sequences $\lambda=\sup_{x\in\omega}\{\lambda\}(x)$ to limit ordinals $\lambda\leq\varepsilon_0$ (in particular $\{\varepsilon_0\}(x)=\omega_{x+1}$ is a tower of $x+1$ exponentials with base $\omega$) we define $F_\alpha$ by recursion on $\alpha\leq\varepsilon_0$, setting
\begin{align*}
F_0(x) &:= x+1,\\
F_{\alpha+1}(x) &:= F_\alpha^{x+1}(x),\\
F_\lambda(x) &:= F_{\{\lambda\}(x)}(x)\quad\text{for $\lambda$ a limit ordinal}.
\end{align*}
To conceive of $\dcon(\pa+\varphi)$ as an arithmetic formula (of complexity $\Pi_1$), recall that ordinals below $\varepsilon_0$ can be represented via their Cantor normal forms. We adopt the efficient encoding of \cite{sommer95}. Building on this one can arithmetize the fast-growing hierarchy: Sommer in \cite[Section 5.2]{sommer95} constructs a $\Delta_0$-formula $F_\alpha(x)=y$ which defines the graphs of the functions $F_\alpha$ for $\alpha\leq\varepsilon_0$ (cf.\ \cite[Equation 4]{freund-proof-length} for the case $\alpha=\varepsilon_0$). Basic relations between these functions become provable in $\isigma_1$. As usual $F_\alpha(x)\!\downarrow$ abbreviates $\exists_y F_\alpha(x)=y$. In addition, the formula $\dcon(\pa+\varphi)$ depends on a formula $\pro_{\isigma_x}(p,\varphi)$ which is $\Delta_1$ in $\isigma_1$ and arithmetizes the ternary relation ``$p$ is a proof of $\varphi$ in the theory $\isigma_x$''. Here $\isigma_x$ denotes the fragment of Peano Arithmetic in which induction is only available for $\Sigma_x$-formulas.\\
It is a classical result, due to Kreisel, Wainer and Schwichtenberg \cite{kreisel52,wainer70,schwichtenberg71}, that Peano Arithmetic does not prove $\forall_x\feps(x)\!\downarrow$. This opens up the possibility that $\dcon$ is strictly weaker than the usual consistency statement. Friedman, Rathjen and Weiermann in \cite{rathjen13} prove that this possibility materializes: Indeed, by \cite[Section 4]{rathjen13} finite iterations of slow consistency generate a strict hierarchy of $\omega$ theories that are stronger than Peano Arithmetic but bounded by the usual consistency statement $\con(\pa)$. It is conjectured in \cite[Remark 4.4]{rathjen13} that the same holds for a transfinite extension of the hierarchy up to any ordinal below $\varepsilon_0$. In the present paper we prove that this is the case: For an appropriate $\Pi_1$-formula $\dcon_\alpha(\pa)$ in the variable $\alpha$ we have
\begin{equation*}
 \pa\subsetneq\cdots\subsetneq\pa+\dcon_\alpha(\pa)\subsetneq\cdots\subsetneq\pa+\dcon_{\varepsilon_0}(\pa)\equiv\pa+\con(\pa).
\end{equation*}
As in \cite[Theorem 3.1]{rathjen13} this is also a strict hierarchy with respect to the interpretability ordering.\\
To prove the result about iterated slow consistency we introduce a notion of slow reflection which is interesting in its own right. As observed by Michael Rathjen in \cite{rathjen-miscellanea-slow-consistency} slow consistency can be derived from a corresponding notion of slow provability, and indeed slow proof: A slow $\pa$-proof of a formula $\varphi$ is a pair $\langle q,\feps(n)\rangle$ such that $q$ is a usual proof of $\varphi$ in the fragment $\isigma_{n+1}$. Writing $\pi_i$ for the projections of the Cantor pairing function this amounts to the formula
\begin{equation*}
 \pro_\pa^\diamond(p,\varphi):\equiv\exists_x(\pro_{\isigma_{x+1}}(\pi_1(p),\varphi)\land\feps(x)=\pi_2(p))
\end{equation*}
which is $\Delta_1$ in $\isigma_1$ (cf.\ \cite[Definition 2.1]{freund-proof-length}). Slow provability is then defined as
\begin{equation*}
 \pr_\pa^\diamond(\varphi):\equiv\exists_p\,\pro_\pa^\diamond(p,\varphi).
\end{equation*} 
Michael Rathjen shows in \cite{rathjen-miscellanea-slow-consistency} that slow provability realizes G\"odel-L\"ob provability logic (see also Lemma \ref{lem:slow-hilbert-bernays} below). It is easy to see that we have
\begin{equation}\label{eq:slow-provability-fragments}
 \isigma_1\vdash\forall_\psi(\pr_\pa^\diamond(\psi)\leftrightarrow\exists_x\,(\feps(x)\!\downarrow\!\land\,\pr_{\isigma_{x+1}}(\psi)))
\end{equation}
and then
\begin{equation*}
 \isigma_1\vdash\dcon(\pa+\varphi)\leftrightarrow\neg\pr_\pa^\diamond(\neg\varphi).
\end{equation*}
Given a notion of provability one can consider the corresponding reflection principles. We will mainly be concerned with uniform reflection. Using Feferman's dot notation, slow (uniform) reflection for the formula $\varphi\equiv\varphi(x_1,\dots ,x_k)$ is defined as
\begin{equation*}
 \drfn(\varphi):\equiv\forall_{x_1,\dots ,x_k}(\pr_\pa^\diamond(\varphi(\dot x_1,\dots ,\dot x_k))\rightarrow\varphi(x_1,\dots ,x_k)).
\end{equation*}
Taking the contrapositive yields the usual connection with iterations of consistency:
\begin{equation*}
 \pa+\drfn(\neg\dcon(\pa+\varphi))\vdash\dcon(\pa+\varphi)\rightarrow\dcon(\pa+\dcon(\pa+\dot\varphi)).
\end{equation*}
We thus need to bound the strength of slow reflection. Consider the set of formulas
\begin{equation*}
 \drfn:=\{\drfn(\varphi)\, |\, \text{$\varphi$ a formula of first-order arithmetic}\}.
\end{equation*}
The central result of this paper is the equiconsistency
\begin{equation}\label{eq:equiconsistency-slow-reflection}
 \pa\vdash\con(\pa)\rightarrow\con(\pa+\drfn).
\end{equation}
Conversely the slow reflection statements are non-trivial: Let $\true_{\Pi_n}(x)$ be the usual truth definition for $\Pi_n$-sentences. We abbreviate
\begin{equation*}
 \drfn(\Pi_n):\equiv\forall_\psi(\text{``$\psi$ a $\Pi_n$-sentence''}\land\pr_\pa^\diamond(\psi)\rightarrow\true_{\Pi_n}(\psi)).
\end{equation*}
Similarly we write $\rfn_\pa(\Pi_n)$ for the usual reflection principles over Peano Arithmetic. We will see that
\begin{equation}\label{eq:slow-reflection-hierarchy-non-trivial}
 \pa+\rfn_\pa(\Pi_{n})\nvdash\drfn(\Pi_{n+1})
\end{equation}
holds for any number $n\geq 1$. An analysis of slow reflection from a more computational viewpoint can be found in \cite[Section 3]{freund-proof-length}: In particular it is shown that $\pa+\drfn(\Pi_2)$ proves the totality of a function $F_{\varepsilon_0}^\diamond$ (a slow variant of $\feps$) which eventually dominates any provably total function of Peano Arithmetic.\\
Of course, one can also consider parameter-free (also called ``local'') slow reflection. Proposition \ref{prop:slow-goryachev} (a slow version of Goryachev's Theorem) links this principle to finite iterations of slow consistency.\\
We should also discuss the issue of index shifts: The original definition of slow consistency in \cite{rathjen13} reads
\begin{equation*}
 \con^*(\pa+\varphi):\equiv\forall_x(\feps(x)\!\downarrow\,\rightarrow\con(\isigma_x+\varphi)),
\end{equation*}
i.e.\ it has $\con(\isigma_x+\varphi)$ where our variant $\dcon(\pa+\varphi)$ demands the stronger $\con(\isigma_{x+1}+\varphi)$. Clearly, the upper bounds that we prove for $\dcon$ also hold for the weaker $\con^*$. It is easy to see that the proof which we will give for the lower bound $\pa+\dcon_{\varepsilon_0}(\pa)\vdash\con(\pa)$ does not depend on the index shift. Interestingly, the results change considerably when we shift the index in the other direction: Set
\begin{equation*}
 \pr_\pa^\dagger(\psi):\equiv\exists_x\,(\feps(x)\!\downarrow\!\land\,\pr_{\isigma_{x+2}}(\psi)))
\end{equation*}
and define $\con^\dagger$ and $\rfn_\pa^\dagger$ accordingly. We will see that $\rfn_\pa^\dagger(\Pi_n)$ is $\pa$-provably equivalent to the usual $\Pi_n$-reflection principle for Peano Arithmetic, for each $n\geq 2$. Concerning slow consistency we will get a hierarchy
\begin{equation*}
 \pa\subsetneq\cdots\subsetneq\pa+\con_n^\dagger(\pa)\subsetneq\cdots\subsetneq\pa+\con_\omega^\dagger(\pa)\equiv\pa+\con(\pa)
\end{equation*}
with only $\omega$ stages below $\con(\pa)$. This justifies that we focus on the $\diamond$-variant: It has the strongest consistency and reflection statements which are non-trivial in the described sense. We refer to \cite{freund-proof-length} for a computational view on the same phenomenon.\\
An independent investigation into slow consistency has been carried out by Paula Henk and Fedor Pakhomov \cite{henk-pakhomov} (for comparison, the first preprint of the present paper was published as arXiv:1601.08214v1). Henk and Pakhomov also prove that the usual consistency statement for Peano Arithmetic is equivalent to $\varepsilon_0$ iterations of slow consistency, and that this goes down to $\omega$ iterations after the index shift. In addition, they construct a ``square root'' consistency statement which reaches ordinary consistency in just two iterations, and they determine the joint provability logic of slow and ordinary provability. They do not consider the notion of slow uniform reflection, which is central to the present paper.

\section{Connecting Reflection and Transfinite Induction}\label{sect:reflection-transfinite-induction}

Using (\ref{eq:slow-provability-fragments}) it is easy to see that we have
\begin{equation}\label{eq:slow-reflection-reflection-fragments}
 \isigma_1\vdash\drfn(\Pi_n)\leftrightarrow\forall_x(\feps(x)\!\downarrow\,\rightarrow\rfn_{\isigma_{x+1}}(\Pi_n))
\end{equation}
for each number $n$. An analogue equivalence characterizes $\drfn(\varphi)$. Let us give a typical application of this equivalence: It is a standard consequence of the ``It's snowing''-Lemma (see \cite[Corollary I.1.76]{hajek91}) that
\begin{equation*}
 \isigma_1\vdash\forall_{x\geq 1}(\rfn_{\isigma_x}(\Pi_n)\rightarrow\rfn_{\isigma_x}(\varphi))
\end{equation*}
holds for any formula $\varphi$ which is $\Pi_n$ in $\isigma_1$. Using (\ref{eq:slow-reflection-reflection-fragments}) we can conclude
\begin{equation*}
 \isigma_1\vdash\drfn(\Pi_n)\rightarrow\drfn(\varphi).
\end{equation*}
Now claim (\ref{eq:slow-reflection-hierarchy-non-trivial}) from the introduction is easily established:

\begin{proposition}
 For any $n\geq 1$ we have
\begin{equation*}
 \pa+\rfn_\pa(\Pi_{n})\nvdash\drfn(\Pi_{n+1}).
\end{equation*}
\end{proposition}
As the proof will show, even a suitable instance of parameter-free slow reflection is unprovable in $\pa+\rfn_\pa(\Pi_{n})$.
\begin{proof}
If we replace slow reflection by the usual reflection principle then the claim is a classical result of Kreisel and L\'evy in \cite{kreisel68}. We combine their proof with an argument specific to slow provability, due to \cite[Proposition 3.3]{rathjen13}: For $n\geq 1$ the formula $\neg\rfn_\pa(\Pi_{n})$ is $\Pi_{n+1}$ in $\isigma_1$. As we have seen above this implies
\begin{equation*}
 \isigma_{1}\vdash\drfn(\Pi_{n+1})\rightarrow\drfn(\neg\rfn_\pa(\Pi_{n}))
\end{equation*}
Aiming at a contradiction, assume that the proposition fails. Then we have
\begin{equation*}
 \isigma_{k+1}+\rfn_\pa(\Pi_n)\vdash\drfn(\neg\rfn_\pa(\Pi_{n}))
\end{equation*}
for  some number $k$. Using an analogue of (\ref{eq:slow-reflection-reflection-fragments}) we can deduce
\begin{equation*}
 \isigma_{k+1}+\rfn_\pa(\Pi_n)\vdash\forall_x(\feps(x)\!\downarrow\,\rightarrow\rfn_{\isigma_{x+1}}(\neg\rfn_\pa(\Pi_{n}))).
\end{equation*}
Since $\isigma_{k+1}$ proves the true $\Sigma_1$-formula $\feps(k)\!\downarrow$ we obtain
\begin{equation*}
 \isigma_{k+1}+\rfn_\pa(\Pi_n)\vdash\rfn_{\isigma_{k+1}}(\neg\rfn_\pa(\Pi_{n})).
\end{equation*}
This is equivalent to
\begin{equation*}
 \isigma_{k+1}+\rfn_\pa(\Pi_n)\vdash\con(\isigma_{k+1}+\rfn_\pa(\Pi_{n})),
\end{equation*}
which contradicts G\"odel's second incompleteness theorem.\\
It is interesting to consider the following alternative argument for the case $n\geq 2$: As is well known $\pa+\rfn_\pa(\Pi_{2})$ proves that $\feps$ is total (cf.\ Lemma \ref{lem:fragments-prove-feps-defined} below). Given that $\feps$ is total, however, the principle $\drfn(\Pi_{n+1})$ collapses into the usual $\rfn_\pa(\Pi_{n+1})$, and we can hark back to the original result of Kreisel and L\'evy \cite{kreisel68}.
\end{proof}

In the rest of this section we reformulate claim (\ref{eq:equiconsistency-slow-reflection}) of the introduction. The goal is to make it accessible for a model construction from \cite{sommer95}, to be carried out in the next section. We begin with an easy observation:

\begin{lemma}[$\isigma_1$]\label{lem:slow-reflection-only-for-pin}
If the theory $\pa+\drfn(\Pi_n)$ is consistent for arbitrarily large $n$ then the theory $\pa+\drfn$ is consistent as well.
\end{lemma}
\begin{proof}
Let $\varphi$ be an arbitrary formula in the language of arithmetic. Choose $n$ such that $\varphi$ is $\isigma_1$-provably equivalent to a $\Pi_n$-formula. We have already shown
\begin{equation*}
 \pa\vdash\drfn(\Pi_n)\rightarrow\drfn(\varphi).
\end{equation*}
This means that $\pa+\drfn$ is contained in $\pa+\{\drfn(\Pi_n)\,|\,n\in\mathbb N\}$. Using (\ref{eq:slow-reflection-reflection-fragments}) we can also show that $m\leq n$ implies
\begin{equation*}
 \pa\vdash\drfn(\Pi_n)\rightarrow\drfn(\Pi_m).
\end{equation*}
By (syntactic) compactness it follows that $\pa+\{\drfn(\Pi_n)\,|\,n\in\mathbb N\}$ is consistent if $\pa+\drfn(\Pi_n)$ is consistent for arbitrarily large $n$.
\end{proof}

Somewhat converse to the proof of the lemma, assume that $\varphi(x)$ is the formula $\true_{\Pi_n}(x)$. Then we have
\begin{equation*}
 \pa\vdash\drfn(\varphi)\rightarrow\drfn(\Pi_n),
\end{equation*}
so that the theories
\begin{equation*}
 \pa+\drfn\,\equiv\,\pa+\{\drfn(\Pi_n)\,|\,n\in\mathbb N\}
\end{equation*}
are equal, as one would expect.\\
We have seen in (\ref{eq:slow-reflection-reflection-fragments}) how slow reflection relates to the usual reflection principles over the fragments of Peano Arithmetic. It is well known that reflection over these fragments corresponds to appropriate instances of transfinite induction (see \cite{kreisel68} for the general idea, and more specifically \cite{ono87} concerning fragments of arithmetic). Since we need to know that this correspondence is available in Peano Arithmetic --- and not only for each fixed fragment but rather uniformly in the fragment $\isigma_{x+1}$ --- we will repeat the arguments in some detail:\\
First, we adopt Sommer's \cite{sommer95} coding of ordinals below $\varepsilon_0$ (observe in particular the notational conventions in \cite[Section~3.4]{sommer95}, which help to distinguish actual ordinals and their numerical codes). Note that the ``stack of $\omega$'s''-function defined by
\begin{equation*}
 \omega_0^\alpha:=\alpha\qquad\omega_{x+1}^\alpha:=\omega^{\omega_x^\alpha}
\end{equation*}
is not part of Sommer's ordinal notation system (although it is part of his meta-theory). We can easily add $(\alpha,x)\mapsto\omega_x^\alpha$ as an $\isigma_1$-provably total function with $\Delta_0$-graph (cf. \cite[Section 2]{freund-proof-length}). As usual we abbreviate $\omega_x:=\omega^1_x$.\\
Next, let us formulate transfinite induction: For a formula $\psi\equiv\psi(\vec x,\gamma)$ with induction variable $\gamma$ and parameters $\vec x$ we set
\begin{align*}
\prog_{\gamma.\psi}(\vec x) & :\equiv\forall_\beta(\forall_{\gamma<\beta}\psi(\vec x,\gamma)\rightarrow\psi(\vec x,\beta)),\\
\ti_{\gamma.\psi}(\alpha) & :\equiv\forall_{\vec x}(\prog_{\gamma.\psi}(\vec x)\rightarrow\forall_{\gamma<\alpha}\psi(\vec x,\gamma)).
\end{align*}
This is similar to the notation used by Feferman \cite[Section~4.3]{feferman-reflections}, who would write $\ti(\alpha,\widehat\gamma\psi(\gamma))$ where we write $\ti_{\gamma.\psi}(\alpha)$. We have decided to move the induction formula to the subscript because we want to reserve the parenthesis for the free variables. Note in particular that $\gamma$ is bound in the formulas $\prog_{\gamma.\psi}(\vec x)$ and $\ti_{\gamma.\psi}(\alpha)$ (the reader may wish to think of $\gamma.\psi(\gamma)$ as a comprehension term with bound variable $\gamma$). Using the truth definition $\true_{\Pi_n}$ for $\Pi_n$-sentences we abbreviate
\begin{equation*}
 \ti_{\Pi_n}(\alpha):\equiv\ti_{\gamma.\true_{\Pi_n}(\varphi(\dot\gamma))}(\alpha).
\end{equation*}
Note that $\true_{\Pi_n}(\varphi(\dot\gamma))$ is a formula with two free variables: The variable $\varphi$, which stands for the code of a $\Pi_n$-formula with a single free variable, is the parameter of the transfinite induction. The induction variable $\gamma$ stands for the code of an ordinal which is substituted for the free variable of $\varphi$. This substitution results in the code of an instance $\varphi(\dot\gamma)$, which $\true_{\Pi_n}(\varphi(\dot\gamma))$ asserts to be true. Note that $\gamma$ becomes bound in the statement $\prog_{\gamma.\true_{\Pi_n}(\varphi(\dot\gamma))}(\varphi)$ (in spite of the dot). The parameter $\varphi$ is free in $\prog_{\gamma.\true_{\Pi_n}(\varphi(\dot\gamma))}(\varphi)$ but also becomes bound in $\ti_{\gamma.\true_{\Pi_n}(\varphi(\dot\gamma))}(\alpha)$. Using the ``It's snowing''-Lemma one establishes
\begin{equation*}
 \isigma_1\vdash\forall_\alpha(\ti_{\Pi_n}(\alpha)\rightarrow\ti_{\gamma.\psi}(\alpha))
\end{equation*}
for any formula $\psi$ which is $\Pi_n$ in $\isigma_1$. Let us connect reflection and transfinite induction:

\begin{proposition}[$\isigma_1$]\label{prop:reflection-transfinite-induction}
 For any number $n\geq 1$ we have
\begin{equation*}
 \isigma_1\vdash\forall_{x\geq n}(\ti_{\Pi_n}(\omega_{x-n+2})\rightarrow\rfn_{\isigma_x}(\Pi_n)).
\end{equation*}
\end{proposition}

A similar result is shown by Ono in \cite[Theorem 4.1]{ono87}. There, however, the quantification over $x$ takes place in the meta-theory, which is not sufficient for our purpose. The following proof (somewhat similar to the proof in \cite{ono87}, but using Buchholz' notations for infinite derivations instead of direct ordinal assignments to finite proofs) shows that the quantification can be internalized.

\begin{proof}
Observe that the proposition itself is a $\Pi_2$-statement: It claims that for each $n\geq 1$ there exists a certain $\isigma_1$-proof. To prove the proposition in the meta-theory $\isigma_1$ we must thus (i) construct the required $\isigma_1$-proofs and (ii)~show that this construction can itself be carried out in $\isigma_1$. Let us focus on task (i) in the first instance. Task (ii) will be discussed below.\\
Fix a number $n\geq 1$ and write it as $n=m+1$. The following proof can be formalized in $\isigma_1$: It is well known that $\rfn_{\isigma_x}(\Pi_n)$ follows from $\rfn_{\isigma_x}(\Sigma_m)$. Consider an arbitrary $x\geq n$. Aiming at $\rfn_{\isigma_x}(\Sigma_m)$, assume that we have $\isigma_x\vdash\varphi$ for some $\Sigma_m$-formula $\varphi$. We suppose that $\isigma_x\vdash\varphi$ is proved in a Tait-style sequent calculus, with induction implemented as a rule
\begin{equation*}
\begin{bprooftree}
\AxiomC{$\Gamma,\psi(0)$}
\AxiomC{$\Gamma,\neg \psi(x),\psi(Sx)$}
\LeftLabel{(Ind)}
\RightLabel{($x$ not free in $\Gamma$).}
\BinaryInfC{$\Gamma,\psi(t)$}
\end{bprooftree}
\end{equation*}
Since we allow arbitrary side formulas the usual induction axioms can be deduced. Partial cut elimination transforms $\isigma_x\vdash\varphi$ into a ``free-cut free'' proof, all cut formulas of which lie in the class $\bigcup_{y\leq x}\Sigma_y\cup\Pi_y$ (see e.g.\ \cite[Section 2.4.6]{buss-introduction-98} for more information).\\
Next, we embed the free-cut free proof $\isigma_x\vdash\varphi$ into an infinite proof system with the $\omega$-rule. To formalize this in the theory $\isigma_1$ we adopt the finitary term system $\mathbf Z^*$ of notations for infinite proofs, developed by Buchholz in \cite{buchholz91} (the reader who is not familiar with these notations will hopefully find enough hints to reconstruct the unformalized argument): Basic terms (constants) of $\mathbf Z^*$ have the form $[d]$ where $d$ is a finite derivation with closed end-sequent. Complex terms are built by the function symbols $I_{k,A}$ (inversion), $R_C$ (cut reduction) and $E$ (cut elimination). Intuitively $[d]$ stands for the infinite proof-tree that results by embedding $d$ into the infinite system, and the function symbols denote the well-known operators from infinite proof theory. Crucially, however, one can work with the term system $\mathbf{Z}^*$ without making the semantics official. Rather, \cite{buchholz91} describes primitive recursive functions which compute the end sequent, the ordinal height, the cut rank, the last rule, and terms denoting the immediate subtrees of an (infinite tree denoted by an) element of $\mathbf{Z}^*$. It is shown that these functions satisfy local correctness conditions, demanding e.g.\ that the immediate subtrees have smaller ordinal height than the whole tree and provide the premises required by the last rule. Let us write $\text{Ord}(h)$, $\text{End}(h)$ and $d_\text{cut}(h)$ for the ordinal height, the end sequent and the cut rank of $h\in\mathbf{Z}^*$. The crucial clauses for us are
\begin{align*}
 \text{Ord}(Eh) & =\exp(\text{Ord}(h)),\\
 \text{End}(Eh) & =\text{End}(h),\\
 d_\text{cut}(Eh) & =d_\text{cut}(h)\dotminus 1.
\end{align*}
To understand the first clause, recall that the ordinal height of an infinite proof grows exponentially when we reduce its cut rank. The most common choice is exponentiation to the base $\omega$. To get better bounds for small ordinals we instead use
\begin{equation*}
 \exp(\alpha):=\begin{cases}
                3^\alpha & \text{if $\alpha<\omega^2$},\\
                \omega^\alpha & \text{otherwise}.
               \end{cases}
\end{equation*}
Semantically one could take $\exp(\alpha)=3^\alpha$ throughout, but then one has to arithmetize ordinal exponentiation to the base $3$. We also remark that exponentiation to the base $2$ would not grow fast enough: The cut elimination operator of \cite[Definition 2.11]{buchholz91} contains an additional step (a repetition rule) to ``call'' the result of cut reduction. A second minor change to \cite{buchholz91} arises from the fact that our finite proofs are formulated with induction rules rather than axioms. Since rules can be nested the embedding lemma now produces ordinals $\text{Ord}([d])<\omega^2$ (instead of $\text{Ord}([d])<\omega\cdot 2$ in the case of induction axioms). Observe that embedding an induction rule produces cuts over the induction formula, but not over formulas of higher complexity. Finally, one attributes cut rank $y$ to formulas in $\Sigma_y\cup\Pi_y$ and cut rank $\infty$ to formulas of a different form. Similar modifications of \cite{buchholz91} can be found in Buchholz' lecture notes \cite{buchholz-lecture-03}.\\
To put this machinery into use, consider the $\Pi_n$-formula
\begin{multline*}
 \text{Sound}_m(\gamma):\equiv\forall_{h\in\mathbf{Z}^*}(\text{Ord}(h)=\gamma\land d_\text{cut}(h)\leq m\land\text{End}(h)\subseteq\Sigma_m\cup\Pi_m\rightarrow\\
\rightarrow\text{``$\text{End}(h)$ contains a true formula (in $\Sigma_m\cup\Pi_m$)''}).
\end{multline*}
It is easy to see that this formula is progressive: The immediate subtrees of $h$ satisfy the assumption of $\text{Sound}_m(\cdot)$ for ordinals $\gamma_n<\gamma$ and they provide the premises to the last rule of $h$ (local correctness of $h\in\mathbf{Z}^*$). Then the induction hypothesis tells us that all the end sequents of the subtrees are true. By Tarski's truth conditions the rules of the infinite proof system are sound for formulas in $\Sigma_m\cup\Pi_m$. It follows that the end sequent of $h$ is true. Having established this, the assumption $\ti_{\Pi_n}(\omega_{x-n+2})$ yields $\forall_{\gamma<\omega_{x-n+2}}\text{Sound}_m(\gamma)$. On the other hand, let $d$ be the above proof $\isigma_x\vdash\varphi$ with cut rank at most $x$. Then the term
\begin{equation*}
 h:=\underbrace{E\cdots\cdots E}_{x-m\text{ symbols}}[d]\in\mathbf{Z}^*
\end{equation*}
 has cut rank at most $m$. Since $\text{Ord}([d])<\omega^2$ implies $\text{Ord}(E[d])=3^{\text{Ord}([d])}<\omega^\omega$ we get $\text{Ord}(h)<\omega_{x-n+2}$. Also, we have $\text{End}(h)=\text{End}([d])=\{\varphi\}$. Thus from $\text{Sound}_m(\text{Ord}(h))$ we can infer that $\varphi$ is true, as required for $\rfn_{\isigma_x}(\Pi_n)$.\\
Recall the tasks (i) and (ii) from the beginning of this proof. So far we have accomplished task (i), i.e.\ for each $n\geq 1$ we have constructed an $\isigma_1$-proof which shows that transfinite $\Pi_n$-induction implies $\Pi_n$-reflection. To settle task (ii) we have to show that the construction of these proofs can itself be formalized in $\isigma_1$. The crucial observation is that the proofs constructed above contain a common core which does not depend on $n$: This is the proof
\begin{equation*}
 \isigma_1\vdash\text{``the term system $\mathbf{Z}^*$ is locally correct''}.
\end{equation*}
Since this core proof is fixed $\isigma_1$ shows that it exists, by $\Sigma_1$-completeness. In the part that does depend on $n$ the main task was to show that the formula $\text{Sound}_m(\gamma)$ is progressive. Besides the core proof, this depended on the fact that $\isigma_1$ proves the Tarski conditions for truth in $\Sigma_m\cup\Pi_m$. A straightforward formalization of \cite[Theorem~I.1.75]{hajek91} shows that these $\isigma_1$-proofs can be constructed in the meta theory~$\isigma_1$.
\end{proof}

Guided by this proposition we introduce the following notion:

\begin{definition}
 For each number $n$ the principle of slow transfinite $\Pi_n$-induction is defined by the formula
\begin{equation*}
 \ti_{\Pi_n}^\diamond:\equiv\forall_{x\geq n\dotminus 1}(\feps(x)\!\downarrow\,\rightarrow\ti_{\Pi_n}(\omega_{x+3-n})).
\end{equation*}
\end{definition}

In the following we are concerned with two goals: We want to connect slow transfinite induction with slow reflection. And we want to show that slow transfinite $\Pi_n$-induction becomes stronger as $n$ grows. First, we need two auxiliary results:

\begin{lemma}[$\isigma_1$]\label{lem:fragments-prove-feps-defined}
 For any number $n$ we have
\begin{equation*}
 \isigma_{n+1}\vdash\feps(n)\!\downarrow.
\end{equation*}
\end{lemma}
\begin{proof}
Since the totality of $\feps$ is not available in the meta-theory $\isigma_1$ (nor even in $\pa$), we cannot simply invoke $\Sigma_1$-completeness. However, $\Sigma_1$-completeness does settle the case $n=0$ (or any finite number of cases). For $n\geq 1$ we recall the following well-known argument: In a weak meta-theory one can formalize the lifting construction for ordinal induction due to Gentzen \cite{gentzen43} (see also \cite[\mbox{Section 4}]{sommer95} concerning fragments of arithmetic). It tells us that $\isigma_{n+1}$ proves transfinite $\Pi_2$-induction up to any ordinal below $\omega_{n+1}$. By \cite[Section 5.2]{sommer95} basic properties of the fast-growing hierarchy are provable in the theory $\isigma_1$. In particular the theory $\isigma_1$ shows that the statement ``$F_\gamma$ is total'' is progressive in $\gamma$. Using ordinal induction $\isigma_{n+1}$ thus proves that $F_{\omega_n^{n+1}}$ is total. Now $\feps(x)\simeq F_{\omega_{x+1}}(x)\simeq F_{\omega_x^{x+1}}(x)$ allows us to conclude the claim. 
\end{proof}

We can deduce the following strengthening of (\ref{eq:slow-reflection-reflection-fragments}):

\begin{lemma}[$\isigma_1$]\label{lem:slow-reflection-fragments-above-bound}
 For any numbers $n$ and $k$ we have
\begin{equation*}
 \isigma_{k+1}\vdash\drfn(\Pi_n)\leftrightarrow\forall_{x\geq k}(\feps(x)\!\downarrow\,\rightarrow\rfn_{\isigma_{x+1}}(\Pi_n)).
\end{equation*}
\end{lemma}
\begin{proof}
Invoking (\ref{eq:slow-reflection-reflection-fragments}) it suffices to show
\begin{equation*}
 \isigma_{k+1}\vdash\forall_{x\geq k}(\feps(x)\!\downarrow\,\rightarrow\rfn_{\isigma_{x+1}}(\Pi_n))\rightarrow\forall_x(\feps(x)\!\downarrow\,\rightarrow\rfn_{\isigma_{x+1}}(\Pi_n)),
\end{equation*}
or also
\begin{equation*}
 \isigma_{k+1}\vdash(\feps(k)\!\downarrow\,\rightarrow\rfn_{\isigma_{k+1}}(\Pi_n))\rightarrow\forall_{x<k}(\feps(x)\!\downarrow\,\rightarrow\rfn_{\isigma_{x+1}}(\Pi_n)).
\end{equation*}
Indeed Lemma \ref{lem:fragments-prove-feps-defined} tells us that we have $\isigma_{k+1}\vdash\feps(k)\!\downarrow$. Now we only need to observe $\isigma_1\vdash \rfn_{\isigma_{k+1}}(\Pi_n)\rightarrow\forall_{x<k}\rfn_{\isigma_{x+1}}(\Pi_n)$.
\end{proof}

Putting pieces together we get the following bound on slow reflection:

\begin{proposition}[$\isigma_1$]\label{prop:slow-induction-implies-sloow-reflection}
 For any number $n\geq 1$ we have
\begin{equation*}
 \isigma_{n+1}\vdash\ti_{\Pi_n}^\diamond\rightarrow\drfn(\Pi_n).
\end{equation*}
\end{proposition}
\begin{proof}
 In view of Lemma \ref{lem:slow-reflection-fragments-above-bound} it suffices to show
\begin{equation*}
 \isigma_1\vdash\ti_{\Pi_n}^\diamond\rightarrow\forall_{x\geq n}(\feps(x)\!\downarrow\,\rightarrow\rfn_{\isigma_{x+1}}(\Pi_n)).
\end{equation*}
This follows from Proposition \ref{prop:reflection-transfinite-induction}.
\end{proof}

The task is now to bound the consistency strength of the theories $\pa+\ti_{\Pi_n}^\diamond$. To prepare this we need yet another auxiliary result:

\begin{lemma}[$\isigma_1$]\label{lem:lifting-ordinal-induction-gentzen}
 For any number $n\geq 1$ we have
\begin{equation*}
 \isigma_1\vdash\forall_{x\geq 1}(\ti_{\Pi_{n+1}}(\omega_x)\rightarrow\ti_{\Pi_n}(\omega_{x+1})).
\end{equation*}
\end{lemma}

This result is of course due to Gentzen \cite{gentzen43}, but again our formulation is somewhat unusual in the way it internalizes $x$. For this reason we recapitulate the proof.

\begin{proof}
We follow Gentzen's construction as presented in \cite[Lemma 4.4]{sommer95}: Consider the lifting formula
\begin{equation*}
 \text{lift}_n(\varphi,\gamma):\equiv\forall_\beta(\forall_{\delta<\beta}\true_{\Pi_n}(\varphi(\dot\delta))\rightarrow\forall_{\delta<\beta+\omega^\gamma}\true_{\Pi_n}(\varphi(\dot\delta))).
\end{equation*}
Note that $\varphi$ is a variable that ranges over codes of formulas, rather than a single fixed formula. Crucially, the form of $\text{lift}_n(\varphi,\gamma)$ depends on $n$ but not on the (non-standard) number $x$. Since $\text{lift}_n(\varphi,\gamma)$ is $\Pi_{n+1}$ in $\isigma_1$ we have
\begin{equation*}
 \isigma_1\vdash\forall_x(\ti_{\Pi_{n+1}}(\omega_x)\rightarrow\ti_{\gamma.\text{lift}_n(\varphi,\gamma)}(\omega_x)).
\end{equation*}
As in the proof of \cite[Lemma 4.4]{sommer95} we have
\begin{equation*}
 \isigma_n\vdash\forall_\varphi(\text{Prog}_{\gamma.\true_{\Pi_n}(\varphi(\dot\gamma))}(\varphi)\rightarrow\text{Prog}_{\gamma.\text{lift}_n(\varphi,\gamma)}(\varphi)).
\end{equation*}
Note that this makes no reference to $x$. Together we obtain
\begin{equation*}
 \isigma_n\vdash\forall_x(\ti_{\Pi_{n+1}}(\omega_x)\rightarrow\forall_\varphi(\text{Prog}_{\gamma.\true_{\Pi_n}(\varphi(\dot\gamma))}(\varphi)\rightarrow\forall_{\gamma<\omega_x}\text{lift}_n(\varphi,\gamma)).
\end{equation*}
Specializing $\beta$ to zero in $\text{lift}_n(\varphi,\gamma)$ we get
\begin{equation*}
 \isigma_n\vdash\forall_x(\ti_{\Pi_{n+1}}(\omega_x)\rightarrow\forall_\varphi(\text{Prog}_{\gamma.\true_{\Pi_n}(\varphi(\dot\gamma))}(\varphi)\rightarrow\forall_{\gamma<\omega_x}\forall_{\delta<\omega^\gamma}\true_{\Pi_n}(\varphi(\dot\delta))).
\end{equation*}
Arguing in $\isigma_1$, if we have $x\geq 1$ then $\omega_x$ is a limit ordinal, so any $\delta<\omega_{x+1}$ is smaller than $\omega^\gamma$ for some $\gamma<\omega_x$. We thus obtain
\begin{equation*}
 \isigma_n\vdash\forall_{x\geq 1}(\ti_{\Pi_{n+1}}(\omega_x)\rightarrow\forall_\varphi(\text{Prog}_{\gamma.\true_{\Pi_n}(\varphi(\dot\gamma))}(\varphi)\rightarrow\forall_{\gamma<\omega_{x+1}}\true_{\Pi_n}(\varphi(\dot\gamma))).
\end{equation*}
Unravelling the abbreviation $\ti_{\Pi_n}(\omega_{x+1})$ we see that this is exactly the same as
\begin{equation*}
 \isigma_n\vdash\forall_{x\geq 1}(\ti_{\Pi_{n+1}}(\omega_x)\rightarrow\ti_{\Pi_n}(\omega_{x+1})).
\end{equation*}
To get the claim of the lemma we need to weaken $\isigma_n$ to $\isigma_1$. This is easy, because the antecedent $\ti_{\Pi_{n+1}}(\omega_x)$ with $x\geq 1$ makes $\Sigma_n$-induction over the natural numbers available.
\end{proof}

Using the lemma we can show that the principle of slow transfinite $\Pi_n$-induction gets stronger as $n$ grows: 

\begin{proposition}[$\isigma_1$]\label{prop:transfinite-induction-higher-formula-complexity}
 For any numbers $0<m\leq n$ we have
\begin{equation*}
 \isigma_n+\ti_{\Pi_n}^\diamond\vdash\ti_{\Pi_m}^\diamond.
\end{equation*}
\end{proposition}
\begin{proof}
 We argue by induction on $n\geq m$. Note that the induction statement is a $\Sigma_1$-formula, as it asserts the existence of a certain proof. For the induction step it suffices to show
\begin{equation*}
 \isigma_{m+1}+\ti_{\Pi_{m+1}}^\diamond\vdash\ti_{\Pi_m}^\diamond,
\end{equation*}
with $m\geq 1$. We argue in $\isigma_{m+1}$: Aiming at $\ti_{\Pi_m}^\diamond$, consider an arbitrary $x\geq m-1$ and assume that $\feps(x)$ is defined. We distinguish two cases: First assume $x\geq m$. Then the assumption $\ti_{\Pi_{m+1}}^\diamond$ yields $\ti_{\Pi_{m+1}}(\omega_{x+3-m-1})$. Using Lemma \ref{lem:lifting-ordinal-induction-gentzen} we get $\ti_{\Pi_m}(\omega_{x+3-m})$, just as required for $\ti_{\Pi_m}^\diamond$. Now assume $x=m-1$. Then we cannot use the assumption $\feps(x)\!\downarrow$, as $\ti_{\Pi_{m+1}}^\diamond$ only speaks about $x\geq m$. Still, $\feps(m)\!\downarrow$ is available by Lemma \ref{lem:fragments-prove-feps-defined}, and we get $\ti_{\Pi_{m}}(\omega_{m+3-m})$ as above. A fortiori we have $\ti_{\Pi_m}(\omega_{x+3-m})$ for $x=m-1$.
\end{proof}

Finally, we obtain the following reformulation of claim (\ref{eq:equiconsistency-slow-reflection}) from the introduction:

\begin{corollary}[$\isigma_1$]\label{cor:consistency-slow-induction-suffices}
 If the theory $\isigma_n+\ti_{\Pi_n}^\diamond$ is consistent for arbitrarily large numbers $n$ then the theory $\pa+\drfn$ is consistent as well.
\end{corollary}
\begin{proof}
By Lemma \ref{lem:slow-reflection-only-for-pin} it is enough to prove that $\pa+\drfn(\Pi_m)$ is consistent for all $m\geq 1$. Proposition \ref{prop:slow-induction-implies-sloow-reflection} reduces this to the consistency of $\pa+\ti_{\Pi_m}^\diamond$. By (syntactic) compactness we only need to show that $\isigma_k+\ti_{\Pi_m}^\diamond$ is consistent for arbitrary $k$. The assumption provides an $n\geq\max\{k,m\}$ such that $\isigma_n+\ti_{\Pi_n}^\diamond$ is consistent. Then it suffices to invoke Proposition \ref{prop:transfinite-induction-higher-formula-complexity}.
\end{proof}

The formulas $\ti_{\Pi_n}^\diamond$ and $\feps(n-1)\!\downarrow$ together entail $\Pi_n$-induction over the natural numbers. It would thus be tempting to replace the theory $\isigma_n+\ti_{\Pi_n}^\diamond$ by $\isigma_1+\ti_{\Pi_n}^\diamond$. However, this is problematic in a weak meta-theory where we do not know that $\isigma_1\vdash\feps(n-1)\!\downarrow$ is true. Still, it will be convenient that the induction formulas of $\isigma_n$ and $\ti_{\Pi_n}^\diamond$ have the same complexity.

\section{Models of Slow Transfinite Induction}\label{sect:bound-slow-induction}

In view of Corollary \ref{cor:consistency-slow-induction-suffices} it remains to show that the theories $\isigma_n+\ti_{\Pi_n}^\diamond$ are consistent. Our approach is inspired by the proof of \cite[Theorem 4.1]{rathjen13}, where a model of Peano Arithmetic is transformed into a model of slow consistency. The main technical ingredient is the construction of models of transfinite induction due to Sommer in \cite[Theorem 5.25]{sommer95}, building on classical work such as \cite{paris80}. As proved originally, \cite[Theorem 5.25]{sommer95} only applies to a standard ordinal $\alpha$. This is most apparent in the proof of \cite[Lemma 5.24]{sommer95}, where $\alpha$ is part of the data encoded in the standard number $m$. Theorem \ref{thm:sommer-modified} below formulates the same result for non-standard ordinals. It turns out that Sommer's proof can be adapted with some modest modifications (see in particular the explanation after Definition \ref{def:alpha-admissible}). Unfortunately, we must review much of Sommer's original proof in order to describe the necessary changes.\\
To formalize the model theoretic arguments as directly as possible it is convenient to work in the subtheory $\aca$ of second-order arithmetic. Recall (e.g.\ from \cite[Theorem III.1.16]{hajek91}) that any first-order theorem of $\aca$ is already provable in Peano Arithmetic. We remark that Sommer in \cite[Section 6.4]{sommer95} formalizes his results in much weaker theories.\\
The general idea of the proof is to start with a model $\mathcal M$ of Peano Arithmetic and to construct an initial segment $\mathcal I\subseteq\mathcal M$ which satisfies some amount of transfinite induction. The segment $\mathcal I$ will be the limit point of a sequence $A$ of elements of $\mathcal M$. The sequence $A$ will be finite in the sense of $\mathcal M$ but it will have non-standard length from an external viewpoint. The satisfaction of $\Sigma_n$-formulas in $\mathcal I$ will be reduced to the satisfaction of corresponding bounded formulas in $\mathcal M$, which we introduce in the following definition. Let us point out that the concepts used in this section are implicit in the proofs from \cite{sommer95}. We find it convenient to extract and name them.

\begin{definition}[cf.\ {\cite[Section 5.5.3]{sommer95}}]\label{def:bounded-variant}
Consider a formula of the form
\begin{equation*}
\varphi(\vec y)\equiv Q^{n-1}_{x_{n-1}}\dots Q^0_{x_0}\theta(\vec x,\vec y),
\end{equation*}
where the $Q^i\in\{\forall,\exists\}$ are all unbounded quantifiers of $\varphi$. Let $z^\varphi=\langle z_0^\varphi,\dots ,z_{n-1}^\varphi\rangle$ be a list of (the first $n$) variables which do not appear in $\varphi$. Then the formula
\begin{equation*}
\varphi^*(\vec y;\vec z\, ^\varphi):\equiv Q^{n-1}_{x_{n-1}\leq z_{n-1}^\varphi}\dots Q^0_{x_0\leq z_0^\varphi}\theta(\vec x,\vec y)
\end{equation*}
is called the bounded variant of $\varphi$.
\end{definition} 

Next, let us characterize the actual bounds $d_0,\dots ,d_{n-1}$ that are to be substituted for the variables $z^\varphi$. They depend on the initial segment $\mathcal I$, or rather on the sequence $A\in\mathcal M$ which has $\mathcal I$ as a limit point. The following definition is to be formalized in Peano Arithmetic: It will be applied inside our model $\mathcal M\vDash\pa$. In particular (iii) refers to the $\pa$-provably total function which maps each code of a formula $\varphi$ (which will be an element of $\mathcal M$) to the code $\varphi^*$ of its bounded variant (as computed in $\mathcal M$).

\begin{definition}[cf.\ {\cite[Lemma 5.11]{sommer95}}]\label{def:n-inductive}
A pair $(A,d)$ is called $n$-inductive if $A=\langle A_0,\dots ,A_{\len(A)-1}\rangle$ codes a strictly increasing sequence, $d=\langle d_0,\dots ,d_{n-1}\rangle$ codes a sequence of length $n$, and the following holds:
\begin{enumerate}[label=(\roman*)]
\item For $i<\len(A)-2$ we have $A_i^2\leq A_{i+1}$.
\item We have $A_{\len(A)-1}\leq d_m$ for all $m<n$.
\item Consider an arbitrary $m<n$, an $i<\len(A)-1$ with $i\geq n-1$, and a $\Pi_m$-formula $\varphi(y_0,\dots ,y_{k-1},x)$. Let $\varphi^*(y_0,\dots ,y_{k-1},x;z_0^\varphi,\dots ,z_{m-1}^\varphi)$ be the bounded variant of $\varphi$, and let $p=\langle p_0,\dots ,p_{k-1}\rangle$ be a list of parameters. If the $4$-tuple $\langle 0,\varphi,p,0\rangle$ has code strictly below $A_i-1$ then the $\Delta_0$-formula
\begin{equation*}
\varphi^*(p_0,\dots ,p_{k-1},w;d_0,\dots ,d_{m-1})
\end{equation*}
is true for some $w\leq A_{i+1}$ if it is true for some $w\leq d_m$.
\end{enumerate}
We say that the $n$-inductive pair $(A,d)$ lies in the interval $[a,b]$ if we have $A_0+n+1\geq a$ and $A_{\len(A)-1}\leq b$.
\end{definition}

The first and last entry of $\langle 0,\varphi,p,0\rangle$ leave room for a future extension. Note that a truth predicate for $\Delta_0$-formulas suffices to formalize the definition. Concerning the encoding of sequences (and tuples, represented as sequences of fixed length), it will be convenient to (provably) have
\begin{gather*}
w\leq s*\langle w\rangle,\\
w\leq w'\rightarrow s*\langle w\rangle\leq s*\langle w'\rangle,\\
\text{``$s'$ is an initial segment of $s$"}\rightarrow s'\leq s.
\end{gather*}
Also, for sequences of any fixed length the code of the sequence should be bounded by a polynomial in its entries. All these requirements hold under the encoding of~\cite[Section 2.2]{sommer95}.\\
Now consider a model $\mathcal M\vDash\pa$ and assume that $A\in\mathcal M$ encodes a strictly increasing sequence (from the viewpoint of $\mathcal M$). Define
\begin{equation*}
A^{\mathcal M}:=\{m\in\mathcal M\, |\, \mathcal M\vDash\exists_{i<\len(A)}\,m=A_i\}.
\end{equation*}
A non-empty initial segment $\mathcal I$ of $\mathcal M$ is called a limit point of $A$ if the set $\mathcal I\cap A^{\mathcal M}$ is unbounded in $\mathcal I$. If the length of $A$ is a non-standard element of $\mathcal M$ then we can define a limit point (indeed the smallest limit point) as
\begin{equation*}
 \mathcal I:=\{m\in\mathcal M\,|\,\text{for some $n\in\mathbb N$ we have }\mathcal M\vDash m\leq A_n\}.
\end{equation*}
While this defines $\mathcal I\subseteq\mathcal M$ as a set, it is not clear how the satisfaction relation for $\mathcal M$ could be transformed into (an arithmetical definition of) a satisfaction relation for $\mathcal I$ (cf.\ \cite[Problem I.4.28]{hajek91}). To circumvent this difficulty, it suffices to read \cite[Lemma 5.11(b)]{sommer95} ``the other way around'', taking the established equivalence as a definition and deducing what is usually the definition of satisfaction in a model:

\begin{proposition}[$\aca$, cf.\ {\cite[Lemma 5.11(b)]{sommer95}}]\label{prop:sigma_n-satisfaction-I}
Consider a standard number~$n$, a model $\mathcal M\vDash\pa$, and a pair $(A,d)\in\mathcal M$ which is $n$-inductive from the viewpoint of $\mathcal M$. Assume that $\mathcal I$ is a limit point of $A$. If $t(\vec x)$ is a term and $\vec p$ are elements of $\mathcal I$ then the value $t(\vec p)^{\mathcal M}$ also lies in $\mathcal I$. Thus
\begin{equation*}
 t(\vec p)^{\mathcal I}:=t(\vec p)^{\mathcal M}
\end{equation*}
defines an interpretation of terms in $\mathcal I$. To interpret the relation symbols $=$ and $\leq$ in $\mathcal I$ one simply restricts their interpretations in $\mathcal M$. To obtain a partial satisfaction relation for $\mathcal I$, consider a formula $\varphi(\vec y)$ in $\bigcup_{m\leq n}\Sigma_m\cup\Pi_m$ and parameters $\vec p\in\mathcal I$, and set
\begin{equation}\label{eq:equivalence-i-m-bounded}
\mathcal I\vDash\varphi(\vec p)\quad:\Leftrightarrow\quad\mathcal M\vDash\varphi^*(\vec p;d_0,\dots ,d_{m-1}).
\end{equation}
This is indeed a partial satisfaction relation for $\mathcal I$, i.e.\ Tarski's conditions hold wherever satisfaction is defined.
\end{proposition}
\begin{proof}
As a limit point, $\mathcal I$ contains all standard elements, and in particular the elements $0,1,2\in\mathcal M$. This reduces closure under successor, addition and multiplication to closure under squaring. Now $\mathcal I$ is closed under squaring by condition (i) of Definition \ref{def:n-inductive}.\\
Next, note that $\varphi^*$ and $\varphi$ are the same formula if $\varphi$ is bounded. Thus we have
\begin{equation}\label{eq:equivalence-i-bounded-m-bounded}
\mathcal I\vDash\varphi(\vec p)\quad\Leftrightarrow\quad\mathcal M\vDash\varphi(\vec p)\qquad\text{for any bounded formula $\varphi$}.
\end{equation}
If the principal connective of a formula $\varphi\in\Sigma_m\cup\Pi_m$ is a propositional connective or a bounded quantifier then $\varphi$ must be a $\Delta_0$-formula. Thus Tarski's conditions for these connectives carry over from $\mathcal M$ (note that, by the first claim, any witness $m\leq t(\vec p)^{\mathcal M}$ to a bounded quantifier lies in $\mathcal I$). The case of an unbounded quantifier relies on the fact that, by clauses (ii) and (iii) of Definition \ref{def:n-inductive}, the quantifier has a witness in $\mathcal I$ if and only if it has a witness below $d_{m-1}$. We refer to the proof of \cite[Lemma 5.11(b)]{sommer95} for more details.
\end{proof}

Let us once more stress the important point that, due to (\ref{eq:equivalence-i-m-bounded}), the complexity of the partial satisfaction relation for $\mathcal I$ does not depend on $n$. It will be convenient to extend the partial satisfaction relation to a fixed number of additional quantifiers: Let $\Pi_3(\Sigma_n)$ be the class of formulas $\forall_{\vec x}\exists_{\vec y}\forall_{\vec z}\psi$ where $\psi$ is a propositional combination of formulas in $\bigcup_{m\leq n}(\Sigma_m\cup\Pi_m)$. Building on the given satisfaction relation for $\Sigma_n$-formulas one can give an arithmetical definition of satisfaction for formulas from the class $\Pi_3(\Sigma_n)$ (just as one usually defines truth for $\Pi_3$-formulas):

\begin{lemma}[$\aca$]\label{lem:extended-satisfaction-I}
 In the situation of Proposition \ref{prop:sigma_n-satisfaction-I}, the partial satisfaction relation over $\mathcal I$ can be extended to a satisfaction relation for $\Pi_3(\Sigma_n)$-formulas, still satisfying Tarski's conditions.
\end{lemma}

Note that the induction axiom for a $\Sigma_n$-formula lies in the class $\Pi_3(\Sigma_n)$ (after prefixing quantifiers). We can thus formulate the following result:

\begin{lemma}[$\aca$]\label{lem:I-satisfies-isigman}
 In the situation of Proposition \ref{prop:sigma_n-satisfaction-I}, the initial segment $\mathcal I$ satisfies all axioms of $\isigma_n$.
\end{lemma}
\begin{proof}
 For all axioms other than induction it suffices to invoke Tarski's conditions and the absoluteness of atomic formulas, provided by (\ref{eq:equivalence-i-bounded-m-bounded}). Equivalence (\ref{eq:equivalence-i-m-bounded}) reduces induction for the $\Sigma_n$-formula $\varphi$ in $\mathcal I$ to induction for the formula $\varphi^*$ in the model $\mathcal M$.
\end{proof}

The usual proof of soundness relies on a full satisfaction relation, and is thus not available for $\mathcal I$. It is standard to fix this:

\begin{lemma}[$\aca$]\label{lem:I-satisfies-consequences}
 In the situation of Proposition \ref{prop:sigma_n-satisfaction-I}, consider two $\Pi_3(\Sigma_n)$-formulas $\varphi(\vec x)$ and $\psi(\vec x)$ and parameters $\vec p\in\mathcal I$. If we have $\isigma_n\vdash\forall_{\vec x}(\varphi\rightarrow\psi)$ then $\mathcal I\vDash\varphi(\vec p)$ implies $\mathcal I\vDash\psi(\vec p)$. In particular, a notion that is $\Delta_1$ in $\isigma_n$ is absolute between $\mathcal I$ and $\mathcal M$.
\end{lemma}
\begin{proof}
First, we transform a given proof of $\varphi\rightarrow\psi$ into a sequent calculus proof (see e.g.\ \cite[Definition 2.3.2]{buss-introduction-98}) of $\Gamma,\varphi\,\Rightarrow\,\psi$, where $\Gamma$ consists of axioms of $\isigma_n$. Next, we eliminate all occurrences of the cut rule (see e.g.\ \cite[Theorem 2.4.2]{buss-introduction-98}). By the subformula property all formulas which occur in the resulting proof belong to the class $\Pi_3(\Sigma_n)$. For this class we have a satisfaction relation, so we can deduce soundness as usual.\\
By Tarski's conditions and the absoluteness of bounded formulas any $\Sigma_1$-formula ($\Pi_1$-formula) is upwards (downwards) absolute. The desired absoluteness follows as, by the first claim, the two versions of a $\Delta_1$-formula are equivalent in the (partial) model $\mathcal I$.
\end{proof}

In particular, the lemma shows how a partial satisfaction relation can yield a consistency result. Before we move on to ordinal induction, let us observe how the initial segment $\mathcal I$ is located inside $\mathcal M$:

\begin{lemma}[$\aca$]\label{lem:location-limit-point}
Adding to the assumptions of Proposition \ref{prop:sigma_n-satisfaction-I}, assume that the $n$-inductive pair $(A,d)$ lies in the interval $[a,b]$ (from the viewpoint of~$\mathcal M$). Then the initial segment $\mathcal I$ contains $a$ but not $b$.
\end{lemma}
\begin{proof}
Straightforward from Definition \ref{def:n-inductive}.
\end{proof}

To see how transfinite induction is accommodated we need some more notions concerning the ordinals below $\varepsilon_0$: Recall (e.g.\ from \cite{sommer95}) that any limit ordinal $\lambda$ is approximated by a strictly increasing ``fundamental'' sequence $(\{\lambda\}(n))_{n\in\mathbb N}$, to be computed from its Cantor normal form. This is extended to non-limit ordinals by the stipulations $\{\alpha+1\}(n)=\alpha$ and $\{0\}(n)=0$. Also recall the concept of an $\alpha$-large sequence: For an ordinal $\alpha$ and a sequence $s=\langle s_0,\dots ,s_{k-1}\rangle$ of natural numbers the ordinal $\{\alpha\}(s)$ is computed by first descending to $\{\alpha\}(s_0)$, then to the $s_1$-th element of the fundamental sequence of that ordinal, finally leading to
\begin{equation*}
 \{\alpha\}(s)=\{\cdots\{\{\alpha\}(s_0)\}(s_1)\cdots\}(s_{k-1}).
\end{equation*}
The sequence $s$ is called $\alpha$-large if we have $\{\alpha\}(s)=0$. It is called exactly $\alpha$-large if it is $\alpha$-large but no proper initial segment of it is $\alpha$-large. According to \cite[5.5.2]{sommer95} the relation $\{\alpha\}(s)=\beta$ is $\Delta_1$ in $\isigma_1$. We will write $s\in S_\alpha$ to express that $s$ is exactly $\omega^\alpha$-large. A finite set will be called (exactly) $\alpha$-large if the strictly increasing sequence which enumerates its elements has that property. The connection with ordinal induction is made by the following result:

\begin{lemma}[$\aca$]\label{lem:ordinals-large-sequences}
There is an $\isigma_1$-provably total function $H_\alpha(\beta)=s$ (in the variables $\alpha$ and $\beta$) such that $\isigma_1$ proves the following: For any $\alpha<\varepsilon_0$ the function $H_\alpha$ restricts to an order-preserving bijection 
\begin{equation*}
 H_\alpha:(\omega_2^\alpha,<)\xrightarrow{\cong} (S_\alpha,<_L),
\end{equation*}
where $<_L$ is the lexicographic ordering of sequences.
\end{lemma}
\begin{proof}
 This is \cite[Theorem 5.12]{sommer95}. Note that the whole statement is $\Sigma_1$, so provability in $\aca$ is no issue.
\end{proof}

Inspired by this correspondence one formulates principles of ``sequence induction'': For a formula $\varphi\equiv\varphi(\vec x,s)$ with induction variable $s$ and parameters $\vec x$ we put
 \begin{alignat*}{3}
& \progs_{s.\varphi}(\vec x,\beta) && :\equiv\forall_{s\in S_\beta}(\forall_{s'\in S_\beta}(s'<_L s\rightarrow\varphi(\vec x,s'))\rightarrow\varphi(\vec x,s)),\\
& \tis_{s.\varphi}(\alpha) && :\equiv\forall_{\vec x}(\progs_{s.\varphi}(\vec x,\alpha)\rightarrow\forall_{s\in S_\alpha}\varphi(\vec x,s)).
\end{alignat*}
Note that being progressive is now relative to an ordinal parameter $\beta$. This is necessary because the correspondence between ordinals and sequences of numbers is not absolute but rather depends on the initial segment of the ordinals in which we are interested. Let us reduce ordinal induction to induction over large sequences:

\begin{lemma}[$\aca$]\label{lem:ordinal-from-sequence-induction}
For any $n\geq 1$ and any $\Pi_n$-formula $\psi(\vec x,\gamma)$ there is a $\Pi_n$-formula $\varphi(\vec x,\delta,s)$ such that we have $\isigma_1\vdash\forall_\alpha(\tis_{s.\varphi}(\alpha)\rightarrow\ti_{\gamma.\psi}(\omega_2^\alpha))$.
\end{lemma}
The crucial point is that the quantification over $\alpha$ occurs in the object theory $\isigma_1$. This is required because we want to use the implication for a non-standard ordinal, as opposed to \cite{sommer95}.
\begin{proof}
Define
\begin{equation*}
\varphi(\vec x,\delta,s):\equiv\forall_{\beta<\omega_2^\delta}(H_\delta(\beta)=s\rightarrow\psi(\vec x,\beta)).
\end{equation*}
Using Lemma \ref{lem:ordinals-large-sequences} one can verify
\begin{equation*}
\isigma_1\vdash\forall_{\vec x}(\prog_{\gamma.\psi}(\vec x)\rightarrow\forall_\alpha\progs_{s.\varphi}(\vec x,\alpha,\alpha)).
\end{equation*}
The claim of the lemma is readily deduced.
\end{proof}

Inspired by these considerations we extend the notion of an $n$-inductive pair, to ensure that the initial segment $\mathcal I$ of $\mathcal M$ satisfies some amount of ordinal induction:

\begin{definition}[cf.\ {\cite[Lemma 5.24]{sommer95}}]\label{def:alpha-inductive}
We say that $A=\langle A_0,\dots ,A_{\len(A)-1}\rangle$ and $d=\langle d_0,\dots ,d_{n-1}\rangle$ form an $(n,\alpha)$-inductive pair if they form an $n$-inductive pair, with $n\geq 1$, and additionally the following holds: Consider an arbitrary number $i<\len(A)-1$, a $\Pi_{n-1}$-formula $\varphi(y_0,\dots ,y_{k-1},s,w)$ with bounded variant $\varphi^*(y_0,\dots ,y_{k-1},s,w;z_0^\varphi,\dots ,z_{n-2}^\varphi)$, and a parameter list $p=\langle p_0,\dots ,p_{k-1}\rangle$. Assume that there is an $s\in S_\alpha$ and a number $w$ such that we have $\langle1,\varphi,p,s*\langle w\rangle\,\rangle< A_i$ and such that $\varphi^*(p_0,\dots ,p_{k-1},s,w;d_0,\dots ,d_{n-2})$ is true. Then there is an $s^0$ which is $<_L$-minimal among the elements of $S_\alpha$ for which the following holds:
\begin{equation}\label{eq:condition-ordinal-induction-lexicographic}
\parbox{9cm}{There is a $w$ such that we have $\langle 1,\varphi,p,s^0*\langle w\rangle\,\rangle<d_{n-1}$ and such that $\varphi^*(p_0,\dots ,p_{k-1},s^0,w;d_0,\dots ,d_{n-2})$ is true.}
\end{equation}
Furthermore, $s^0$ has code below $A_{i+1}$ and the corresponding instance of (\ref{eq:condition-ordinal-induction-lexicographic}) holds with a witness $w$ that is smaller than $A_{i+1}$.
\end{definition}

Let us show that a limit point of an $(n,\alpha)$-inductive pair satisfies a certain amount of transfinite induction.

\begin{proposition}[$\aca$, cf.\ {\cite[Theorem 5.25]{sommer95}}]\label{prop:I-satisfies-trasnfinite-induction}
In the situation of Proposition \ref{prop:sigma_n-satisfaction-I}, assume that $(A,d)$ is $(n,\alpha)$-inductive for some ordinal $\alpha\in\mathcal M$ (all in the sense of $\mathcal M$), and with $n\geq 1$. If $\alpha$ lies in $\mathcal I$ then we have $\mathcal I\vDash\ti_{\Pi_n}(\omega_2^\alpha)$.
\end{proposition}

Note that $\ti_{\Pi_n}(\omega_2^\alpha)$ is $\Pi_3(\Sigma_n)$ after prefixing quantifiers (cf.\ Lemma \ref{lem:extended-satisfaction-I}).

\begin{proof}
The proof is essentially that of \cite[Theorem 5.25]{sommer95}, but we repeat it to demonstrate the functioning of our terminology. The crucial difference to \cite{sommer95} is that $\alpha$ may now be non-standard. By Lemma \ref{lem:ordinal-from-sequence-induction} and Lemma \ref{lem:I-satisfies-consequences} it suffices to verify the sequence induction principle $\mathcal I\vDash\tis_{s.\psi}(\alpha)$ for an appropriate $\Pi_n$-formula $\psi$. Write $\psi\equiv\psi(\vec y,s)\equiv\forall_w\varphi(\vec y,s,w)$. By Tarski's conditions for $\mathcal I$ it suffices to show the following: Given arbitrary parameters $\vec p\in\mathcal I$, assume that for some $s\in\mathcal I$ we have (i) $\mathcal I\vDash s\in S_\alpha$ and (ii) $\mathcal I\nvDash\psi(\vec p,s)$. We must find a $<_L$-minimal $s^0\in\mathcal I$ which satisfies (i) and (ii).\\
From (ii) we infer that $\mathcal I\nvDash\varphi(\vec p,s,w)$ holds for some $w\in\mathcal I$. Write $\widetilde\varphi$ for the $\Pi_{n-1}$-formula which results from $\neg\varphi$ by pulling the negation under the unbounded quantifiers. Thus we have $\mathcal I\vDash\widetilde\varphi(\vec p,s,w)$, and Proposition \ref{prop:sigma_n-satisfaction-I} transforms this into $\mathcal M\vDash\widetilde\varphi^*(\vec p,s,w;d_0,\dots ,d_{n-2})$. In view of the ``It's snowing''-Lemma $\mathcal M$ also satisfies the statement ``$\widetilde\varphi^*(\vec p,s,w;d_0,\dots ,d_{n-2})$ is true''. Let us verify the other assumptions from Definition \ref{def:alpha-inductive}: Because of $n\geq 1$ any formula which is $\Delta_1$ in $\isigma_1$ is absolute between $\mathcal I$ and $\mathcal M$. Thus (i) yields $\mathcal M\vDash s\in S_\alpha$. Next, note that $\widetilde\varphi$ is a standard formula and that the parameter list $\vec p$ has standard length. Thus (the code of) $\widetilde\varphi$ lies in $\mathcal I$, and so does the tuple $\langle 1,\widetilde\varphi,\vec p,s*\langle w\rangle\,\rangle$. Since $\mathcal I$ is a limit point of $A$ we have $\langle 1,\widetilde\varphi,\vec p,s*\langle w\rangle\,\rangle<A_i$ for some $i<\len(A)-1$ with $A_i\in\mathcal I$. Let $s^0$ be provided by Definition~\ref{def:alpha-inductive}. From $s^0\leq A_{i+1}$ we infer $s^0\in\mathcal I$. Absoluteness of $\Delta_1$-formulas gives $\mathcal I\vDash s^0\in S_\alpha$, which is (i) above. Also, Definition \ref{def:alpha-inductive} tells us that $\mathcal M\vDash\widetilde\varphi^*(\vec p,s^0,w;d_0,\dots ,d_{n-2})$ holds for some $w\leq A_{i+1}$, i.e.\ $w$ lies in $\mathcal I$. We conclude $\mathcal I\vDash\widetilde\varphi(\vec p,s^0,w)$ and then $\mathcal I\nvDash\varphi(\vec p,s^0,w)$, as required for (ii) above. The $<_L$-minimality of $s^0$ in $\mathcal I$ is similarly deduced from the minimality provided by Definition \ref{def:alpha-inductive}.
\end{proof}

Next, we show that Peano Arithmetic proves the existence of $(n,\alpha)$-inductive pairs, under the assumption that certain large sets exist. This will allow us to apply Proposition \ref{prop:I-satisfies-trasnfinite-induction}. We need an auxiliary notion:

\begin{definition}\label{def:alpha-admissible}
For $n\geq 1$, an $n$-inductive pair $(A,d)$ is called $\alpha$-admissible if the following holds: Consider $i<\len(A)-1$, a $\Pi_{n-1}$-formula $\varphi(y_0,\dots ,y_{k-1},s,x)$ with bounded variant $\varphi^*(y_0,\dots ,y_{k-1},s,x;z_0^\varphi,\dots ,z_{n-2}^\varphi)$, a list $p=\langle p_0,\dots ,p_{k-1}\rangle$ of parameters, and a sequence $s'$. If we have $\langle 1,\varphi,p,s'\rangle<A_i-1$ then the statement
\begin{equation}\label{eq:alpha-admissible}
 \parbox{10cm}{``the sequence $s$ is exactly $\omega^\alpha$-large, $s'$ is an initial segment of $s$, and the $\Delta_0$-formula $\varphi^*(p_0,\dots ,p_{k-1},s,w;d_0,\dots ,d_{n-2})$ is true''}
\end{equation}
holds for some $s,w$ with $\langle 1,\varphi,p,s*\langle w\rangle\,\rangle<A_{i+1}$ if it holds for some $s,w$ with $\langle1,\varphi,p,s*\langle w\rangle\,\rangle<d_{n-1}$.
\end{definition}

We remark that the proof of \cite[Lemma 5.24]{sommer95} treats Definition \ref{def:alpha-admissible} as a special case of Definition \ref{def:n-inductive}(iii). To do so, one pulls the truth predicate around the whole statement (\ref{eq:alpha-admissible}), such that this whole statement becomes the formula $\varphi$ of Definition \ref{def:n-inductive}. Then, however, the ordinal $\alpha$ becomes part of the parameter list $p$, and the condition $\langle 0,\varphi,p,0\rangle<A_i-1$ forces us to consider bounds on its code. This is no problem if $\alpha$ is standard, as in \cite{sommer95}, but it becomes an issue when we consider non-standard ordinals. The notion of $\alpha$-admissibility disentangles $\alpha$ and $p$, and then \cite{sommer95} extends to non-standard ordinals:

\begin{proposition}[$\pa$, cf.\ {\cite[Lemma 5.10, 5.11(a)]{sommer95}}]\label{prop:construct-n-inductive}
Assume that we have $F_{\omega_{n-1}^\gamma}(a)=b$ for ordinals $\gamma\geq 1$ and $\alpha$ and numbers $n\geq 1$, $a\geq n+1$ and $b$. Then there is an $n$-inductive $\alpha$-admissible pair $(A,d)$ in $[a,b]$, such that $A$ is $\gamma$-large.
\end{proposition}

\begin{proof}
 The proof is by induction on $n$, and very similar to the proofs of \cite[Lemma 5.10, 5.11(a)]{sommer95}. Let us review the case $n=1$. The idea is to build sequences $A$ and $B$ with the following properties:
\begin{enumerate}[label=(\roman*')]
\item For $i<\len(A)-1$ we have $F_{\{\gamma\}\langle A_0,\dots ,A_i\rangle}(A_i)\leq A_{i+1}$.
\item We have $a-2\leq A_0<\dots <A_{\len(A)-1}\leq B_{\len(B)-1}\leq\dots\leq B_0\leq b$, as well as $\len(A)=\len(B)$.
\item Consider $i<\len(A)-1$, a bounded formula $\varphi(y_0,\dots ,y_{k-1},x)$ and a parameter list $p=\langle p_0,\dots ,p_{k-1}\rangle$ with $\langle 0,\varphi,p,0\rangle<A_i-1$. Then the formula $\varphi(p_0,\dots ,p_{k-1},w)$ is true for some $w\leq A_{i+1}$ if it is true for some $w\leq B_{i+1}$.
\item Consider $i<\len(A)-1$, a bounded formula $\varphi(y_0,\dots ,y_{k-1},s,x)$, a parameter list $p=\langle p_0,\dots ,p_{k-1}\rangle$ and a sequence $s'$ with $\langle 1,\varphi,p,s'\rangle<A_i-1$. Then statement (\ref{eq:alpha-admissible}) holds for some $s,w$ with $\langle 1,\varphi,p,s*\langle w\rangle\,\rangle+1\leq A_{i+1}$ if it holds for some $s,w$ with $\langle 1,\varphi,p,s*\langle w\rangle\,\rangle+1\leq B_{i+1}$.
\item For any $i<\len(A)$ the set $[A_i,B_{i}]\cap [a-2,b]$ is $(\omega^{\{\gamma\}\langle A_0,\dots ,A_{i-1}\rangle}+1)$-large.
\end{enumerate}
By induction on $l$ we prove the following:
\begin{equation}\label{eq:claim-construct-step-inductive-sequences}
\parbox{11cm}{There are sequences $A$ and $B$, both of length $l+1$, such that the following holds for all $j\leq l$: Either $\langle A_0,\dots ,A_{j-1}\rangle$ is $\gamma$-large or the sequences $\langle A_0,\dots ,A_j\rangle$ and $\langle B_0,\dots ,B_j\rangle$ fulfill conditions (i') to (v').}
\end{equation}
In the base case $l=0$ we set $A:=\langle a-2\rangle$ and $B:=\langle b\rangle$. Condition (ii') is immediate, and conditions (i',iii',iv') are void because of $\len(A)-1=0$. For condition (v') we need to check that the set $[a-2,b]$ is $(\omega^\gamma+1)$-large. This follows from the assumption $F_\gamma(a)=b$ by \cite[Proposition 5.8]{sommer95}. We come to the induction step $l\leadsto l+1$: Let $A$ and $B$ be the sequences given by the induction hypothesis. We may assume that $A=\langle A_0,\dots ,A_l\rangle$ is not $\gamma$-large. In particular, $\langle A_0,\dots ,A_{l-1}\rangle$ is not $\gamma$-large, and the induction hypothesis tells us that $A$ and $B$ satisfy (i') to (v'). By condition (v') the set $[A_l,B_l]\subseteq [a-2,b]$ is $(\omega^{\{\gamma\}\langle A_0,\dots ,A_{l-1}\rangle}+1)$-large. Using \cite[Proposition 5.9]{sommer95} we can write
\begin{equation*}
 [A_l,B_l]=\{A_l,A_l +1\}\sqcup I^1\sqcup\dots\sqcup I^{A_l+1}
\end{equation*}
as a disjoint union, such that each interval $I^i$ is $\omega^{\{\gamma\}\langle A_0,\dots ,A_l\rangle}$-large (thus in particular non-empty). To extend $A$ and $B$ as required for the induction step, we would like to pick $A_{l+1}:=\max(I^i)$ and $B_{l+1}:=\max(I^{i+1})$ for some $1\leq i\leq A_l$. Conditions (iii') and (iv') will be satisfied if certain minimal witnesses do not lie in the interval $I^{i+1}$. The assumptions $\langle 0,\varphi,p,0\rangle<A_l-1$ and $\langle 1,\varphi,p,s'\rangle<A_l-1$ ensure that there are less than $A_l$ relevant witnesses. Thus we can pick a suitable interval by the pigeonhole principle. The construction also validates (ii') and~(v'). Condition (i') follows since $[A_l+1,A_{l+1}]$ is $\omega^{\{\gamma\}\langle A_0,\dots ,A_l\rangle}$-large, using \cite[Proposition 5.8, 5.9]{sommer95}.\\
Now let $A$ and $B$ be provided by (\ref{eq:claim-construct-step-inductive-sequences}), for $l=b-a+3$. Since $A$ cannot satisfy condition (ii') it must be $\gamma$-large. Shortening the sequences if necessary, we can assume that $A$ is exactly $\gamma$-large. Still by (\ref{eq:claim-construct-step-inductive-sequences}) it follows that (the shortened) $A$ and $B$ satisfy conditions (i') to (v'). We set $d:=\langle A_{\len(A)-1}\rangle$. It is immediate that $(A,d)$ has most properties of a $1$-inductive $\alpha$-admissible pair in $[a,b]$. What remains to be checked is that $A_i^2\leq A_{i+1}$ holds for all $i<\len(A)-2$. For such an $i$ we must have $\{\gamma\}\langle A_0,\dots ,A_i\rangle\geq 2$, and then (i') combined with \cite[Proposition~5.4]{sommer95} allows us to conclude.\\
For the induction step $n\leadsto n+1$ one argues similarly, replacing the interval $[a-2,b]$ by the set $\{a-(n+2),A_0,A_1,\dots ,A_{\len(A)-1}\}$ given by the induction hypothesis. We refer to the proof of \cite[Lemma 5.11(a)]{sommer95} for details. 
\end{proof}

Building on this, we can construct $(n,\alpha)$-inductive pairs:

\begin{proposition}[$\pa$]\label{prop:alpha-admissible-gives-alpha-inductive}
For numbers $n\geq 1$ and $c$ and an ordinal $\alpha$, consider an $(\omega^\alpha\cdot c)$-large sequence $A$ such that $(A,d)$ is an $n$-inductive $\alpha$-admissible pair in the interval $[a,b]$. Then there is an $(n,\alpha)$-inductive pair $(B,d)$ in $[a,b]$ with~$\len(B)=c$.
\end{proposition}

The central idea of the following proof is due to \cite[Lemma 5.24]{sommer95}, but we must adjust the details in a non-trivial way: The original statement of the result assumes an inequality $a\geq m+n+1$ where $m$ depends on (the code of) $\alpha$. We have to avoid such a bound, since in our application $\alpha$ will itself depend on the (non-standard) number $a$.

\begin{proof}
By \cite[Proposition 5.9]{sommer95} the sequence $A$ can be written as a concatenation $A=A^0*\dots *A^{c-1}$ such that each sequence $A^j$ is $\omega^\alpha$-large. Let $B_j$ be the first element of the sequence $A^j$, and set $B:=\langle B_0,\dots ,B_{c-1}\rangle$. It is easy to see that the pair $(B,d)$ is $n$-inductive and lies in the interval $[a,b]$, as $B$ is a subsequence of $A$. To verify that it is $(n,\alpha)$-inductive, consider \mbox{$j<\len(B)-1$}, a $\Pi_{n-1}$-formula $\varphi(y_0,\dots ,y_{k-1},s,x)$, a parameter list $p=\langle p_0,\dots ,p_{k-1}\rangle$, a sequence $s\in S_\alpha$ and a number $w$ with $\langle 1,\varphi,p,s*\langle w\rangle\,\rangle<B_j$, such that $\varphi^*(p_0,\dots ,p_{k-1},s,w;d_0,\dots ,d_{n-2})$ is true. We construct a sequence $s^0$ with $\len(s^0)=\len(A^j)$ as follows:
\begin{equation}\label{eq:construct-minimal-sequence}
 \parbox{11cm}{For any $i<\len(s^0)$, either $\langle s^0_0,\dots ,s^0_{i-1}\rangle$ is $\omega^\alpha$-large or else $s^0_i$ is minimal with the following property: There is an end-extension $s\in S_\alpha$ of $\langle s^0_0,\dots ,s^0_i\rangle$ and a number $w$ with $(1,\varphi,p,s*\langle w\rangle )<d_{n-1}$ such that $\varphi^*(p_0,\dots ,p_{k-1},s,w;d_0,\dots ,d_{n-2})$ is true.}
\end{equation}
Indeed, if $\langle s^0_0,\dots ,s^0_{i-1}\rangle$ is not $\omega^\alpha$-large then the induction hypothesis or the above assumptions (in case $i=0$) provide us with an end-extension $s$ of $\langle s^0_0,\dots ,s^0_{i-1}\rangle$ as described in (\ref{eq:construct-minimal-sequence}). Since $s$ is $\omega^\alpha$-large it is even an end-extension of $\langle s^0_0,\dots ,s^0_{i-1},x\rangle$ for some $x$. It suffices to minimize over $x$ to get $s^0_i$. Now that the construction is complete, let us establish the following property of $s^0$:
\begin{equation*}
 \parbox{11cm}{If $\langle s^0_0,\dots ,s^0_{i-1}\rangle$ is not $\omega^\alpha$-large then $\langle 1,\varphi,p,\langle s^0_0,\dots ,s^0_i\rangle\,\rangle<A^j_i-1$ holds.}
\end{equation*}
We argue by induction on $i$. The base case $i=0$ follows from the above assumption $\langle 1,\varphi,p,s*\langle w\rangle\rangle<B_j=A^j_0$ and the inequalities listed after Definition \ref{def:n-inductive}. In the step the induction hypothesis tells us $\langle 1,\varphi,p,\langle s^0_0,\dots ,s^0_i\rangle\,\rangle<A^j_i-1$. We can combine this with (\ref{eq:construct-minimal-sequence}) and the fact that $A$ is $\alpha$-admissible, to learn that $\langle s^0_0,\dots ,s^0_i\rangle$ has an end-extension $s\in S_\alpha$ such that the formula $\varphi^*(p_0,\dots ,p_{k-1},s,w;d_0,\dots ,d_{n-2})$ is true for some $w$ with $\langle 1,\varphi,p,s*\langle w\rangle\,\rangle<A^j_{i+1}$. Minimality of $s^0_{i+1}$ yields the required
\begin{equation*}
 \langle 1,\varphi,p,\langle s^0_0,\dots ,s^0_{i+1}\rangle\,\rangle\leq\langle 1,\varphi,p,s\rangle<\langle 1,\varphi,p,s*\langle w\rangle\,\rangle\leq A^j_{i+1}-1.
\end{equation*}
In particular we can infer $s^0_i\leq A^j_i$. As $A^j$ is $\omega^\alpha$-large it follows that $s^0$ is $\omega^\alpha$-large, by \cite[Proposition 5.9]{sommer95}. Possibly after shortening the sequence $s^0$ we may assume that it is exactly $\omega^\alpha$-large. When we apply (\ref{eq:construct-minimal-sequence}) with $i=\len(s^0)-1$ we obtain a sequence $s$ which is exactly $\omega^\alpha$-large and has $s^0$ as an initial segment. This forces $s^0=s$, so that $s^0$ has the properties attributed to $s$ in (\ref{eq:construct-minimal-sequence}), i.e.\ it satisfies condition (\ref{eq:condition-ordinal-induction-lexicographic}) from the definition of an $(n,\alpha)$-inductive sequence. We have already seen $\langle 1,\varphi,p,s^0\rangle<A^j_{\len(A^j)-1}-1$. Note that $A^{j+1}_0$ follows $A^j_{\len(A^j)-1}$ in the list $A$. By $\alpha$-admissibility we obtain a $w$ such that $\varphi^*(p_0,\dots ,p_{k-1},s^0,w;d_0,\dots ,d_{n-2})$ is true and such that we have
\begin{equation*}
 \langle 1,\varphi,p,s^0*\langle w\rangle\,\rangle<A^{j+1}_0=B_{j+1}.
\end{equation*}
This implies $s^0<B_{j+1}$ and $w<B_{j+1}$, as required by Definition \ref{def:alpha-inductive}. The $<_L$-minimality of $s^0$ follows easily from its construction.
\end{proof}

Putting pieces together we obtain a non-standard version of Sommer's result:

\begin{theorem}[$\aca$; cf.\ {\cite[Theorem 5.25]{sommer95}}]\label{thm:sommer-modified}
Consider $n\geq 1$, a model $\mathcal M\vDash\pa$, a non-standard number $c\in\mathcal M$, and an ordinal $\alpha\in\mathcal M$ (possibly non-standard). Assume that we have
\begin{equation*}
\mathcal M\vDash F_{\omega_{n-1}^{\omega^\alpha\cdot c}}(a)=b
\end{equation*}
for some elements $a\geq n+1$ and $b$ of $\mathcal M$. Then there is an initial segment $\mathcal I\subseteq\mathcal M$ equipped with a satisfaction relation for $\Pi_3(\Sigma_n)$-formulas. We have $a\in\mathcal I$ and $b\notin\mathcal I$, and $\mathcal I$ satisfies all axioms of $\isigma_n$. If $\alpha$ lies in $\mathcal I$ then we have $\mathcal I\vDash\ti_{\Pi_n}(\omega_2^\alpha)$.
\end{theorem}
\begin{proof}
 Proposition \ref{prop:construct-n-inductive} yields a pair $(A,d)\in\mathcal M$ in $[a,b]$ which is $n$-inductive and $\alpha$-admissible, and such that $A$ is $(\omega^\alpha\cdot c)$-large (all from the viewpoint of $\mathcal M$). By Proposition \ref{prop:alpha-admissible-gives-alpha-inductive} this can be transformed into an $(n,\alpha)$-inductive pair $(B,d)\in\mathcal M$ in $[a,b]$, with $B$ a sequence of length $c$. Since $c$ is non-standard the sequence $B$ has a limit point $\mathcal I$. Proposition \ref{prop:sigma_n-satisfaction-I} and Lemma \ref{lem:extended-satisfaction-I} equip $\mathcal I$ with a satisfaction relation for $\Pi_3(\Sigma_n)$-formulas. By Lemma \ref{lem:location-limit-point} we have $a\in\mathcal I$ and~$b\notin\mathcal I$. Lemma \ref{lem:I-satisfies-isigman} yields $\mathcal I\vDash\isigma_n$, and Proposition \ref{prop:I-satisfies-trasnfinite-induction} accounts for transfinite induction.
\end{proof}

To apply this theorem we need a preparatory result:

\begin{lemma}[$\pa$]\label{lem:smaller-than-feps-for-slow-induction-model}
 For any $y<x$, if $\feps(x)$ is defined then $F_{\omega_y^{\omega_{x-y}\cdot (x+1)}}(x)$ is defined and has value at most $\feps(x)$.
\end{lemma}
\begin{proof}
 Recall the ``step down"-relation $\beta\! \searrow_x\! \alpha$ of \cite{solovay81}, arithmetized in \cite{rathjen13}: It expresses that $\alpha$ can be reached from $\beta$ by descending to the $x$-th member of the fundamental sequence a finite number of times. It is well known that, provably in $\pa$, we have $\alpha\!\searrow_x\! 0$ for any ordinal $\alpha<\varepsilon_0$ (cf.\ \cite[Proposition~2.9]{solovay81}). Also recall $\feps(x)\simeq F_{\omega_{x+1}}(x)$. Thus \cite[Lemma 2.3]{rathjen13} reduces the claim to \mbox{$\omega_{x+1}\! \searrow_x\! \omega_y^{\omega_{x-y}\cdot(x+1)}$}. To establish the latter, note that $\omega_{x-y}\! \searrow_x\! \omega_{x-y-1}+1$ holds by \cite[Lemma 2.13]{rathjen13}. Then \cite[Lemma 2.10]{rathjen13} gives $\omega_{x-y+1}\! \searrow_x\! \omega^{\omega_{x-y-1}+1}$. In view of $\{\omega^{\omega_{x-y-1}+1}\}(x)=\omega_{x-y}\cdot (x+1)$ we conclude $\omega_{x-y+1}\! \searrow_x\! \omega_{x-y}\cdot (x+1)$. Iterating \cite[Lemma 2.10]{rathjen13} yields $\omega_{x+1}\! \searrow_x\! \omega_y^{\omega_{x-y}\cdot(x+1)}$, as desired.
\end{proof}

Finally, we can construct a model of slow transfinite induction. Our proof is inspired by that of \cite[Theorem 4.1]{rathjen13}, which uses the standard version of Sommer's result.

\begin{proposition}[$\pa$]\label{prop:consistency-slow-induction}
If Peano Arithmetic is consistent then so is the theory $\isigma_n+\ti_{\Pi_n}^\diamond$, for each $n\geq 1$.
\end{proposition}
\begin{proof}
 Since the claim is arithmetic we can work in $\aca$. The assumption that Peano Arithmetic is consistent provides us with a non-standard model $\mathcal M\vDash\pa$. By Lemma~\ref{lem:fragments-prove-feps-defined} we have $\mathcal M\vDash\feps(k)\!\downarrow$ for each standard number~$k$. It is worth noting that this step requires $\mathcal M$ to be a model of full Peano Arithmetic. Using overspill we get
 \begin{equation*}
 \mathcal M\vDash\feps(a)=b\qquad\text{for some non-standard numbers $a,b\in\mathcal M$}.
 \end{equation*}
Lemma \ref{lem:smaller-than-feps-for-slow-induction-model} tells us that $F_{\omega_{n-1}^{\omega_{a-n+1}\cdot (a+1)}}(a)$ is defined in $\mathcal M$, and that its value is at most $b$. Now apply Theorem \ref{thm:sommer-modified} with $\alpha=\omega_{a-n}$ (as computed in $\mathcal M$). This produces an initial segment $\mathcal I\subseteq\mathcal M$, equipped with a satisfaction relation for $\Pi_3(\Sigma_n)$-formulas, such that we have $a\in\mathcal I$, $b\notin\mathcal I$ and $\mathcal I\vDash\isigma_n$. Recall that $x\mapsto\omega_x$ (in terms of codes) is an $\isigma_1$-provably total function with $\Delta_0$-graph. Thus we may compute $\omega_{a-n}$ from the viewpoint of $\mathcal I$. By absoluteness the computations in $\mathcal I$ and $\mathcal M$ yield the same ordinal $\omega_{a-n}=\alpha$. Theorem \ref{thm:sommer-modified} thus also gives $\mathcal I\vDash\ti_{\Pi_n}(\omega_{a-n+2})$. Recall
\begin{equation*}
 \ti_{\Pi_n}^\diamond\equiv\forall_{x\geq n\dotminus 1}(\feps(x)\!\downarrow\,\rightarrow\ti_{\Pi_n}(\omega_{x+3-n})).
\end{equation*}
To verify $\mathcal I\vDash\ti_{\Pi_n}^\diamond$ we consider an arbitrary $p\in\mathcal I$ with $p\geq n-1$. Assume that we have $\mathcal I\vDash\feps(p)=q$ for some $q\in\mathcal I$, and observe that this implies $q<b$. Since $\feps(x)=y$ is a bounded formula we also have $\mathcal M\vDash\feps(p)=q$. Using \cite[Lemma 2.3]{rathjen13} we can conclude $p<a$. Thus $\mathcal I\vDash\ti_{\Pi_n}(\omega_{a-n+2})$ implies $\mathcal I\vDash\ti_{\Pi_n}(\omega_{p+3-n})$, as required for $\mathcal I\vDash\ti_{\Pi_n}^\diamond$. Even though $\mathcal I$ does not have a full satisfaction relation, the desired consistency result follows by Lemma \ref{lem:I-satisfies-consequences}.
\end{proof}

The promised result about slow uniform reflection follows immediately:

\begin{theorem}\label{thm:bound-consistency-slow-reflection}
 We have
\begin{equation*}
 \pa\vdash\con(\pa)\rightarrow\con(\pa+\drfn).
\end{equation*}
\end{theorem}
\begin{proof}
 Combine Corollary \ref{cor:consistency-slow-induction-suffices} and Proposition \ref{prop:consistency-slow-induction}. 
\end{proof}

\section{Iterated Slow Consistency}\label{sect:iterated-consistency}

Recall the slow consistency formula
\begin{equation*}
\dcon(\pa+\varphi)\equiv\forall_x(\feps(x)\!\downarrow\,\rightarrow\con(\isigma_{x+1}+\varphi))
\end{equation*}
with variable $\varphi$. Following \cite[Lemma 2.1]{beklemishev03}, transfinite iterations along the ordinals can be defined with the help of the fixed point theorem (see e.g.\ \cite[Theorem III.2.1]{hajek91}): It provides a formula $\dcon_\alpha(\pa)$ with variable $\alpha$ such that we have
\begin{equation}\label{eq:def-iterated-consistency}
 \isigma_1\vdash\dcon_\alpha(\pa)\leftrightarrow\forall_{\beta<\alpha}\dcon(\pa+\dcon_{\dot\beta}(\pa)).
\end{equation}
In particular the equivalence implies that $\dcon_\alpha(\pa)$ is $\Pi_1$ in $\isigma_1$. Even though $\varepsilon_0$ is not part of the ordinal notation system, the equivalence makes the statement $\dcon_{\varepsilon_0}(\pa)$ meaningful. To show that iterations of slow consistency become stronger as the ordinal parameter grows we need the following observation:

\begin{lemma}[$\Sigma_1$-completeness for slow provability]\label{lem:slow-sigma1-completeness}
 If $\varphi\equiv\varphi(x)$ is $\Sigma_1$ in $\isigma_n$, with $n\geq 1$, then we have
\begin{equation*}
 \isigma_n\vdash\varphi(x)\rightarrow\pr^\diamond_\pa(\varphi(\dot x)).
\end{equation*}
\end{lemma}
\begin{proof}
 Write $n=k+1$. By assumption and $\Sigma_1$-completeness for the usual notion of proof we have
\begin{equation*}
 \isigma_n\vdash\varphi(x)\rightarrow\pr_{\isigma_{k+1}}(\varphi(\dot x))
\end{equation*}
In view of $\isigma_n\vdash\feps(k)\!\downarrow$ the result follows by (\ref{eq:slow-provability-fragments}).
\end{proof}

We can deduce that the hierarchy of slow consistency statements is strict:

\begin{proposition}\label{prop:strict-hierarchy-slow-consistency}
 We have
\begin{equation*}
 \isigma_1\vdash\forall_{\beta<\alpha}(\dcon_\alpha(\pa)\rightarrow\dcon_\beta(\pa)).
\end{equation*}
Given ordinals $\beta<\alpha\leq\varepsilon_0$ we have
\begin{equation*}
 \pa+\dcon_\beta(\pa)\nvdash\dcon_\alpha(\pa).
\end{equation*}
\end{proposition}
\begin{proof}
 Equivalence (\ref{eq:def-iterated-consistency}) reduces the first claim to
\begin{equation}\label{eq:contrapositive-sigma1-completeness}
 \isigma_1\vdash\dcon(\pa+\dcon_{\dot\beta}(\pa))\rightarrow\dcon_\beta(\pa).
\end{equation}
This is nothing but the contrapositive of slow $\Sigma_1$-completeness. Let us come to the second claim: Aiming at a contradiction, assume that $\dcon_\alpha(\pa)$ is provable in $\pa+\dcon_\beta(\pa)$. By equivalence (\ref{eq:def-iterated-consistency}) this implies
\begin{equation*}
 \pa+\dcon_\beta(\pa)\vdash\dcon(\pa+\dcon_\beta(\pa)).
\end{equation*}
According to \cite[Corollary 3.4]{rathjen13} this is only possible if $\pa+\dcon_\beta(\pa)$ is inconsistent. The consistency of $\pa+\dcon_\beta(\pa)$ is easily established in a strong meta-theory: By induction on $\beta$ one shows that $\dcon_\beta(\pa)$ holds in the standard model, using (\ref{eq:def-iterated-consistency}) for the induction step. To avoid a strong meta-theory we could invoke Theorem~\ref{thm:lower-bound-con-in-slow-com} below (which does not rely on the present claim): It shows that $\pa+\con(\pa)$ proves the consistency of $\pa+\dcon_\beta(\pa)$.
\end{proof}

The following result shows in particular that the finite part of the hierarchy coincides with the iterations considered in \cite[Section 4]{rathjen13}.

\begin{proposition}\label{prop:iterations-succ-limit}
 We have
\begin{align*}
 \isigma_1&\vdash\dcon_0(\pa),\\
 \isigma_1&\vdash\dcon_{\alpha+1}(\pa)\leftrightarrow\dcon(\pa+\dcon_{\dot\alpha}(\pa)),\\
 \isigma_1&\vdash\text{``$\lambda$ limit''}\rightarrow(\dcon_\lambda(\pa)\leftrightarrow\forall_{\gamma<\lambda}\dcon_\gamma(\pa)).
\end{align*}
\end{proposition}
\begin{proof}
 The first claim follows immediately from (\ref{eq:def-iterated-consistency}), and so does direction ``$\rightarrow$'' of the second claim. For the other direction, assume $\dcon(\pa+\dcon_{\dot\alpha}(\pa))$. By (\ref{eq:contrapositive-sigma1-completeness}) we get $\dcon_\alpha(\pa)$, i.e.\ $\forall_{\beta<\alpha}\dcon(\pa+\dcon_{\dot\beta}(\pa))$. Together we have $\forall_{\beta<\alpha+1}\dcon(\pa+\dcon_{\dot\beta}(\pa))$, and thus $\dcon_{\alpha+1}(\pa)$ by (\ref{eq:def-iterated-consistency}). The last claim follows from similar considerations.
\end{proof}

We remark that all theories $\pa+\dcon_\alpha(\pa)$ in our hierarchy are finite extensions of Peano Arithmetic. In this respect our set-up differs from the iterations of (usual) consistency investigated by Schmerl \cite{schmerl79} and Beklemishev \cite{beklemishev03}: Their approach would suggest to consider the infinite extension
\begin{equation*}
 \pa+\{\dcon_\gamma(\pa)\,|\,\gamma<\lambda\}
\end{equation*}
at limit stage $\lambda$. We want to avoid infinite extensions because they make the notion of slow consistency somewhat less canonic: Assuming $\feps(x)\!\downarrow$, should we demand the consistency of $\isigma_{x+1}+\{\dcon_\gamma(\pa)\,|\,\gamma<\lambda\}$ or the consistency of, say, $\isigma_{x+1}+\{\dcon_\gamma(\pa)\,|\,\gamma<\{\lambda\}(x+1)\}$? Another reason is that we hope to reach the finite extension $\pa+\con(\pa)$ at limit stage $\varepsilon_0$.\\
We have seen that iterations of slow consistency generate a strict hierarchy. Now we prove a main result of this paper, relating this hierarchy to the usual consistency statement:

\begin{theorem}\label{thm:lower-bound-con-in-slow-com}
 We have
\begin{equation*}
 \pa+\con(\pa)\vdash\forall_{\alpha<\varepsilon_0}\con(\pa+\dcon_{\dot\alpha}(\pa)).
\end{equation*}
By $\Sigma_1$-completeness and Proposition \ref{prop:iterations-succ-limit} this implies
\begin{equation*}
 \pa+\con(\pa)\vdash\dcon_{\varepsilon_0}(\pa).
\end{equation*}
\end{theorem}

It is worth noting that the proof uses slow reflection only for $\Sigma_1$-formulas.

\begin{proof}
 By Theorem \ref{thm:bound-consistency-slow-reflection} we have $\pa+\con(\pa)\vdash\con(\pa+\drfn)$, so it suffices to establish
\begin{equation*}
 \pa\vdash\forall_{\alpha<\varepsilon_0}\pr_{\pa+\drfn}(\dcon_{\dot\alpha}(\pa)).
\end{equation*}
As famously shown by Gentzen \cite{gentzen43}, $\pa$ proves ordinal induction for any proper initial segment of $\varepsilon_0$. This fact can itself be established in $\pa$ (cf.\ Lemma \ref{lem:lifting-ordinal-induction-gentzen}). The open claim is thus reduced to
\begin{equation*}
 \pa\vdash\pr_{\pa+\drfn}(\prog_{\alpha.\dcon_{\alpha}(\pa)}),
\end{equation*}
i.e.\ the theory $\pa+\drfn$ must prove that $\dcon_{\alpha}(\pa)$ is progressive in the ordinal parameter $\alpha$. Considering the contrapositive of $\drfn(\neg\dcon_{\dot\beta}(\pa))$ we have
\begin{equation*}
 \pa+\drfn\vdash\forall_\gamma(\dcon_\gamma(\pa)\rightarrow\dcon(\pa+\dcon_{\dot\gamma}(\pa))).
\end{equation*}
Together with (\ref{eq:def-iterated-consistency}) we get
\begin{equation*}
 \pa+\drfn\vdash\forall_\beta(\forall_{\gamma<\beta}\dcon_\gamma(\pa)\rightarrow\dcon_\beta(\pa)),
\end{equation*}
just as required.
\end{proof}

Our next goal is to prove the converse implication, namely that $\varepsilon_0$ iterations of slow consistency yield the usual consistency statement. This follows from a result of Schmerl \cite{schmerl79} and Beklemishev \cite{beklemishev03}, stating that $\con(\pa)$ is implied by $\varepsilon_0$ iterations of consistency over the elementary arithmetic $\ea$. To conclude one only needs to observe that slow consistency entails the consistency of $\ea$, and that this is preserved under iterations. Since $\feps(0)\!\downarrow$ is provable by $\Sigma_1$-completeness equivalence (\ref{eq:def-slow-consistency}) shows that slow consistency even implies $\con(\isigma_1)$. The same argument works for the variant $\con^*$ (because $\feps(1)\!\downarrow$ is provable as well, cf.\ the introduction) and other possible variations of slow consistency. Let us provide details: We have already mentioned that \cite{beklemishev03} works with infinite collections of consistency statements, taking
\begin{equation*}
 \ea_\alpha:=\ea+\{\con(\ea_\beta)\,|\,\beta<\alpha\}
\end{equation*}
for the theory at stage $\alpha$. Still, the fixed point theorem allows us to express the relation $\ea_\alpha\vdash\varphi$: Following \cite[Equation 3]{beklemishev03} we consider a formula $\Box(\alpha,\varphi)$ with
\begin{equation}\label{eq:def-iterations-beklemishev}
 \ea\vdash\Box(\alpha,\varphi)\leftrightarrow\pr_{\ea+\{\neg\Box(\dot\beta,0=1)\,|\,\beta<\alpha\}}(\varphi).
\end{equation}
By \cite[Proposition 7.3, Remark 7.4]{beklemishev03} we have
\begin{equation*}
 \isigma_1\vdash\text{``$\varphi$ a $\Pi_1$-formula''}\land\pr_\pa(\varphi)\rightarrow\Box(\varepsilon_0,\varphi).
\end{equation*}
Taking the contrapositive of the instance $\varphi\equiv (0=1)$ yields
\begin{equation}\label{eq:epsilon-beklemishev-implies-consistency}
 \isigma_1\vdash\neg\Box(\varepsilon_0,0=1)\rightarrow\con(\pa).
\end{equation}
Let us establish the connection with slow consistency:

\begin{lemma}\label{lem:slow-consistency-stronger-consistency-beklemishev}
 We have
\begin{equation*}
\isigma_1\vdash\forall_{\alpha<\varepsilon_0}(\dcon_{\alpha+1}(\pa)\rightarrow\neg\Box(\alpha,0=1)).
\end{equation*}
\end{lemma}
\begin{proof}
Let us abbreviate
\begin{equation*}
 \psi(\alpha):\equiv\con(\isigma_1+\dcon_{\dot\alpha}(\pa))\rightarrow\neg\Box(\alpha,0=1).
\end{equation*}
In view of (\ref{eq:def-iterated-consistency}), (\ref{eq:def-slow-consistency}) and $\isigma_1\vdash\feps(0)\!\downarrow$ we have
\begin{equation*}
\isigma_1\vdash\forall_\alpha(\dcon_{\alpha+1}(\pa)\rightarrow\con(\isigma_1+\dcon_{\dot\alpha}(\pa))).
\end{equation*}
Thus it suffices to prove $\isigma_1\vdash\forall_{\alpha<\varepsilon_0}\psi(\alpha)$. We would like to argue by transfinite induction, but this does not yield a uniform proof for all ordinals below $\varepsilon_0$ (not even if we work in $\pa$ instead of $\isigma_1$). Instead we use a trick based on L\"ob's theorem, namely the ``reflexive induction rule'' introduced by Schmerl in \cite{schmerl79} (see also the proof of \cite[Lemma 3.2]{beklemishev03}). This admissible rule allows us to conclude $\isigma_1\vdash\forall_\alpha\psi(\alpha)$ once we have established
\begin{equation}\label{eq:for-reflexive-induction-rule-slow-implies-ea}
 \isigma_1\vdash\pr_{\isigma_1}(\forall_{\beta<\dot\alpha}\psi(\beta))\rightarrow\psi(\alpha).
\end{equation}
To prove the latter, let us work in $\isigma_1$: We assume $\pr_{\isigma_1}(\forall_{\beta<\dot\alpha}\psi(\beta))$ and, unravelling the definition of $\psi$, also $\con(\isigma_1+\dcon_{\dot\alpha}(\pa))$. By (\ref{eq:def-iterated-consistency}) this gives
\begin{equation*}
 \con(\isigma_1+\forall_{\beta<\dot\alpha}\dcon(\pa+\dcon_{\dot\beta}(\pa))).
\end{equation*}
Note how the dot-notation operates on two different levels. By $\Sigma_1$-completeness $\feps(0)\!\downarrow$ is available. Thus (\ref{eq:def-slow-consistency}) yields
\begin{equation*}
 \con(\isigma_1+\forall_{\beta<\dot\alpha}\con(\isigma_1+\dcon_{\dot\beta}(\pa))).
\end{equation*}
Using the assumption $\pr_{\isigma_1}(\forall_{\beta<\dot\alpha}\psi(\beta))$, i.e.\ the reflexive induction hypothesis, we conclude
\begin{equation*}
 \con(\isigma_1+\forall_{\beta<\dot\alpha}\neg\Box(\beta,0=1)).
\end{equation*}
A fortiori we have
\begin{equation*}
 \con(\ea+\{\neg\Box(\dot\beta,0=1)\,|\,\beta<\alpha\}).
\end{equation*}
By (\ref{eq:def-iterations-beklemishev}) we get $\neg\Box(\alpha,0=1)$, i.e.\ the conclusion of $\psi(\alpha)$, as required for (\ref{eq:for-reflexive-induction-rule-slow-implies-ea}).
\end{proof}

Now we can deduce the converse bound to Theorem \ref{thm:lower-bound-con-in-slow-com}:

\begin{corollary}
 We have
\begin{equation*}
 \isigma_1+\dcon_{\varepsilon_0}(\pa)\vdash\con(\pa).
\end{equation*}
\end{corollary}
\begin{proof}
We argue in $\isigma_1$: By (\ref{eq:def-iterated-consistency}) and Proposition \ref{prop:iterations-succ-limit} the assumption $\dcon_{\varepsilon_0}(\pa)$ implies $\forall_{\alpha<\varepsilon_0}\dcon_{\alpha+1}(\pa)$. Lemma \ref{lem:slow-consistency-stronger-consistency-beklemishev} yields $\forall_{\alpha<\varepsilon_0}\neg\Box(\alpha,0=1)$. Using (\ref{eq:def-iterations-beklemishev}) we arrive at $\forall_{\alpha<\varepsilon_0}\con(\ea+\{\neg\Box(\dot\beta,0=1)\,|\,\beta<\alpha\})$, and (the syntactic version of) compactness leads to $\con(\ea+\{\neg\Box(\dot\beta,0=1)\,|\,\beta<\varepsilon_0\})$. The other direction of (\ref{eq:def-iterations-beklemishev}) gives $\neg\Box(\varepsilon_0,0=1)$. Finally $\con(\pa)$ follows from (\ref{eq:epsilon-beklemishev-implies-consistency}).
\end{proof}

We come back to the topic of index shifts in the definition of slow provability, as discussed in the introduction. Recall the $\dagger$-variant with consistency statement
\begin{equation}\label{eq:consistency-dagger-variant}
 \isigma_1\vdash\con^\dagger(\pa+\varphi)\leftrightarrow\forall_x(\feps(x)\!\downarrow\,\rightarrow\con(\isigma_{x+2}+\varphi))
\end{equation}
and uniform reflection principles
\begin{equation*}
 \isigma_1\vdash\rfn_\pa^\dagger(\Pi_n)\leftrightarrow\forall_x(\feps(x)\!\downarrow\,\rightarrow\rfn_{\isigma_{x+2}}(\Pi_n)).
\end{equation*}
As already observed in \cite{freund-proof-length} this index shift makes slow reflection as powerful as the usual reflection principle:

\begin{lemma}
 For all $n\geq 2$ we have
\begin{equation*}
 \isigma_1\vdash\rfn_\pa^\dagger(\Pi_n)\rightarrow\rfn_\pa(\Pi_n).
\end{equation*}
\end{lemma}
\begin{proof}
 It suffices to show that $\rfn_\pa^\dagger(\Pi_n)$ implies the totality of $\feps$. Let us prove the induction step
\begin{equation*}
 \isigma_1+\rfn_\pa^\dagger(\Pi_n)\vdash\feps(x)\!\downarrow\,\rightarrow\feps(x+1)\!\downarrow.
\end{equation*}
Working in $\isigma_1$, the assumptions $\rfn_\pa^\dagger(\Pi_n)$ and $\feps(x)\!\downarrow$ make $\Pi_n$-reflection over the theory $\isigma_{x+2}$ available. On the other hand, Lemma \ref{lem:fragments-prove-feps-defined} tells us that $\isigma_{x+2}$ proves the $\Sigma_1$-formula $\feps(x+1)\!\downarrow$. 
\end{proof}

The situation is somewhat different for $n=1$: Finitely many iterations of slow consistency are still weaker than the usual consistency statement. Of course, iterations of $\con^\dagger$ are defined parallel to (\ref{eq:def-iterated-consistency}). We will cite results for the $\diamond$-variant if the proofs carry over easily.

\begin{proposition}\label{prop:con-stronger-omega-daggers}
 We have
\begin{equation*}
 \pa\vdash\con(\pa)\rightarrow\con_\omega^\dagger(\pa).
\end{equation*}
\end{proposition}
\begin{proof}
 According to Proposition \ref{prop:iterations-succ-limit} we can replace $\con_\omega^\dagger(\pa)$ by $\forall_n\con_n^\dagger(\pa)$. Invoking $\Sigma_1$-completeness the claim can then be strengthened to
\begin{equation*}
 \pa+\con(\pa)\vdash\forall_n\con(\pa+\con_{\dot n}^\dagger(\pa)).
\end{equation*}
This is established by induction on $n$. Concerning the base case, Proposition \ref{prop:iterations-succ-limit} gives $\pa\vdash\con_0^\dagger(\pa)$. Thus $\con(\pa+\con_0^\dagger(\pa))$ follows from the assumption $\con(\pa)$. Again using Proposition \ref{prop:iterations-succ-limit} the induction step amounts to
\begin{equation*}
 \pa\vdash\con(\pa+\con_{\dot n}^\dagger(\pa))\rightarrow\con(\pa+\con^\dagger(\pa+\con_{\dot n}^\dagger(\pa))).
\end{equation*}
Note that, by $\Sigma_1$-completeness, it does not matter whether the code of the formula $\con_{\dot n}^\dagger(\pa)$ in the conclusion is computed in the object theory (i.e.\ inside the consistency statement) or in the meta-theory.  Then the induction step follows from the more general statement
\begin{equation*}
 \pa\vdash\forall_\varphi(\con(\pa+\varphi)\rightarrow\con(\pa+\con^\dagger(\pa+\dot\varphi))).
\end{equation*}
This is essentially \cite[Theorem 4.1]{rathjen13}. Let us repeat the argument given there, to show that one can accommodate the index shift: Working in $\pa$, the assumption $\con(\pa+\varphi)$ provides a non-standard model $\mathcal M\vDash\pa+\varphi$. Since $\pa+\varphi$ is reflexive we have $\mathcal M\vDash\con(\isigma_{k+2}+\varphi)$ for each number $k$. By overspill we get
\begin{equation*}
 \mathcal M\vDash\con(\isigma_{a+2}+\varphi)\qquad\text{for some non-standard $a\in\mathcal M$}.
\end{equation*}
Assume first that we have $\mathcal M\vDash\feps(a+1)\!\downarrow$, i.e.\ that $\mathcal M\vDash\feps(a+1)=b$ holds for some $b\in\mathcal M$. Let $n$ be a standard number. By \cite[Corollary 3.8]{rathjen13} (derived from \cite[Theorem 2.25]{sommer95}, see also our Theorem \ref{thm:sommer-modified}) there is an initial segment $\mathcal I\subseteq\mathcal M$ with $a\in\mathcal I$, $b\notin\mathcal I$ and $\mathcal I\vDash\isigma_{n+1}$. We verify that $\mathcal I$ satisfies $\con^\dagger(\pa+\varphi)$ as characterized by (\ref{eq:consistency-dagger-variant}): Assume that we have $\mathcal I\vDash\feps(p)=q$ for some $p,q\in\mathcal I$. This implies $\mathcal M\vDash\feps(p)=q$, and then $q<b$ yields $p\leq a$. Thus we have $\mathcal M\vDash\con(\isigma_{p+2}+\varphi)$. Since this is a $\Pi_1$-formula we get $\mathcal I\vDash\con(\isigma_{p+2}+\varphi)$, as required for $\mathcal I\vDash\con^\dagger(\pa+\varphi)$. Since $n$ was arbitrary we have indeed shown $\con(\pa+\con^\dagger(\pa+\varphi))$. In the remaining case, assume $\mathcal M\nvDash\feps(a+1)\!\downarrow$. Then $\mathcal M$ itself satisfies $\con^\dagger(\pa+\varphi)$: Consider $p\in\mathcal M$ with $\mathcal M\vDash\feps(p)\!\downarrow$. In view of \cite[Lemma~2.3]{rathjen13} we must have $p\leq a$. Thus $\mathcal M\vDash\con(\isigma_{a+2}+\varphi)$ implies $\mathcal M\vDash\con(\isigma_{p+2}+\varphi)$, as required for $\mathcal M\vDash\con^\dagger(\pa+\varphi)$. As $\mathcal M$ is a model of~$\pa$ we have again established the consistency of $\pa+\con^\dagger(\pa+\varphi)$.
\end{proof}

Conversely, the index shift makes $\omega$ iterations of slow consistency as strong as the usual consistency statement:

\begin{proposition}
 We have
\begin{equation*}
 \isigma_1\vdash\con^\dagger_\omega(\pa)\rightarrow\con(\pa).
\end{equation*}
\end{proposition}
\begin{proof}
By Proposition \ref{prop:iterations-succ-limit} the assumption $\con^\dagger_\omega(\pa)$ implies $\forall_{n<\omega}\con^\dagger_n(\pa)$. Thus it suffices to prove
\begin{equation*}
 \isigma_1\vdash\forall_n(\con_{n+1}^\dagger(\pa)\rightarrow\forall_{m\leq n}\con(\isigma_{m+1}+\con_{\overline{n-m}}^\dagger(\pa))).
\end{equation*}
Note that $\overline{n-m}$ refers to the dot notation. Working in $\isigma_1$, consider an arbitrary $n$ and assume $\con_{n+1}^\dagger(\pa)$. We show the conclusion by induction on $m$. For the base case $m=0$, note that the assumption $\con_{n+1}^\dagger(\pa)$ implies $\con^\dagger(\pa+\con_{\dot n}^\dagger(\pa))$. In view of (\ref{eq:consistency-dagger-variant}) and $\feps(0)\!\downarrow$ we get the desired $\con(\isigma_1+\con_{\dot n}^\dagger(\pa))$. In the step $m\leadsto m+1$ the induction hypothesis provides $\con(\isigma_{m+1}+\con_{\overline{n-m}}^\dagger(\pa))$. Invoking (\ref{eq:def-iterated-consistency}) this implies
\begin{equation*}
 \con(\isigma_{m+1}+\con^\dagger(\pa+\con_{\overline{n-(m+1)}}^\dagger(\pa))).
\end{equation*}
Crucially, Lemma \ref{lem:fragments-prove-feps-defined} provides $\isigma_{m+1}\vdash\feps(m)\!\downarrow$. In view of (\ref{eq:consistency-dagger-variant}) we can conclude
\begin{equation*}
 \con(\isigma_{m+1}+\con(\isigma_{\dot m+2}+\con_{\overline{n-(m+1)}}^\dagger(\pa))).
\end{equation*}
The desired $\con(\isigma_{m+2}+\con_{\overline{n-(m+1)}}^\dagger(\pa))$ follows by $\Sigma_1$-completeness.
\end{proof}

Finally, we consider parameter-free slow reflection. Notationally, we switch back to the $\diamond$-variant, but it is easy to see that the same arguments apply to the other versions of slow consistency. Parameter-free slow reflection for a closed formula $\varphi$ is defined as
\begin{equation*}
 \prfn^\diamond_\pa(\varphi):\equiv\pr_\pa^\diamond(\varphi)\rightarrow\varphi.
\end{equation*}
It will be useful to know that slow provability behaves, in important respects, like the usual provability predicate:

\begin{lemma}[Slow Hilbert-Bernays conditions; {\cite[Lemma~5.3]{rathjen-miscellanea-slow-consistency}}]\label{lem:slow-hilbert-bernays}
If $\varphi$ is provable in Peano Arithmetic then so is $\pr_\pa^\diamond(\varphi)$. Furthermore the following holds:
\begin{align*}
\isigma_1 & \vdash\forall_\varphi(\pr_{\pa}^\diamond(\varphi)\rightarrow\pr_{\pa}^\diamond(\pr_{\pa}^\diamond(\dot\varphi))),\\
\isigma_1 & \vdash\forall_{\varphi,\psi}(\pr_{\pa}^\diamond(\varphi\rightarrow\psi)\land\pr_\pa^\diamond(\varphi)\rightarrow\pr_\pa^\diamond(\psi)).
\end{align*}
\end{lemma}

Additionally, Rathjen in \cite[Lemma~5.3]{rathjen-miscellanea-slow-consistency} proves a slow version of the formalized L\"ob theorem.

\begin{proof}
If Peano Arithmetic proves $\varphi$ then we have $\isigma_{k+1}\vdash\varphi$ for some $k$. By the usual provability condition we get $\pa\vdash\pr_{k+1}(\varphi)$. Together with $\pa\vdash\feps(k)\!\downarrow$ we can conclude $\pa\vdash\pr_\pa^\diamond(\varphi)$ by (\ref{eq:slow-provability-fragments}). Next, work in $\isigma_1$ and assume $\pr_{\pa}^\diamond(\varphi)$. By $\Sigma_1$-completeness we obtain $\pr_{\isigma_1}(\pr_{\pa}^\diamond(\dot\varphi))$, and in view of $\feps(0)\!\downarrow$ equivalence (\ref{eq:slow-provability-fragments}) yields $\pr_\pa^\diamond(\pr_{\pa}^\diamond(\dot\varphi))$. For the last claim we again work in $\isigma_1$. Assume $\pr_{\pa}^\diamond(\varphi\rightarrow\psi)$ and $\pr_{\pa}^\diamond(\varphi)$. By (\ref{eq:slow-provability-fragments}) we get numbers $x$ and $y$ with $\pr_{\isigma_{x+1}}(\varphi\rightarrow\psi)$ and $\pr_{\isigma_{y+1}}(\varphi)$, as well as $\feps(x)\!\downarrow$ and $\feps(y)\!\downarrow$. We assume $x\geq y$, the other case being symmetric. Then the usual provability condition provides $\pr_{\isigma_{x+1}}(\psi)$, and $\pr_\pa^\diamond(\psi)$ follows by (\ref{eq:slow-provability-fragments}).
\end{proof}

The following consequence is due to a hint by Michael Rathjen:

\begin{proposition}[Slow Goryachev's Theorem]\label{prop:slow-goryachev}
We have
\begin{equation*}
 \pa+\{\prfn_\pa^\diamond(\varphi)\,|\,\text{$\varphi$ a closed formula}\}\quad\equiv_{\Pi_1}\quad\pa+\{\dcon_n(\pa)\,|\,n<\omega\},
\end{equation*}
i.e.\ the two theories prove the same $\Pi_1$-sentences.
\end{proposition}
\begin{proof}
Concerning the inclusion $\supseteq$, the contrapositive of the local reflection principle $\prfn_\pa^\diamond(\neg\dcon_n(\pa))$ is the implication
\begin{equation*}
 \dcon_n(\pa)\rightarrow\dcon(\pa+\dcon_n(\pa)).
\end{equation*}
So inductively all iterations $\dcon_n(\pa)$ can be deduced from the local reflection principles. Using slow $\Sigma_1$-completeness (see Lemma \ref{lem:slow-sigma1-completeness}) the inclusion $\subseteq_{\Pi_1}$ is easily reduced to the following claim: For any closed formulas $\varphi_1,\dots,\varphi_n$ we have
\begin{equation*}
 \pa\vdash\dcon_{n+1}(\pa)\rightarrow\dcon\left(\pa+\bigwedge_{i=1,\dots ,n}\prfn^\diamond_\pa(\varphi_i)\right).
\end{equation*}
This claim is established as in the case of ordinary provability: A detailed argument can be found in \cite[Section 14.1]{franzen04}. Since it only uses the Hilbert-Bernays conditions it immediately applies to slow provability. Another syntactical proof of Goryachev's theorem can be found in \cite[Theorem~IV.5]{lindstroem97}. Alternatively, Goryachev's theorem can be established by a semantical argument, using the completeness of G\"odel-L\"ob provability logic for Kripke models: Such a proof can be found in \cite[Lemma~4.2]{beklemishev-iterated-local-reflection}. By \cite[Theorem~5.4]{rathjen-miscellanea-slow-consistency} slow provability models G\"odel-L\"ob logic, so the semantical proof of Goryachev's theorem applies to slow provability as well.
\end{proof}

It is easy to check that the same result holds for the $\dagger$-variant of slow provability. Together with Proposition \ref{prop:con-stronger-omega-daggers} it follows that the usual consistency statement for Peano Arithmetic is not provable in $\pa+\{\prfn_\pa^\dagger(\varphi)\,|\,\text{$\varphi$ a closed formula}\}$. Thus, as opposed to the situation for uniform reflection, the index shift does not make parameter-free slow reflection trivial.\\
The relationship between parameter-free (i.e.~local) reflection for formulas of different complexities has been investigated by Beklemishev \cite{beklemishev-notes-local-reflection} (I am grateful to the referee for this hint): First, it is important to observe that for $n\geq 2$ the theories
\begin{equation*}
 \pa+\prfn_\pa(\Pi_n):\equiv\pa+\{\prfn_\pa(\varphi)\,|\,\text{$\varphi$ a closed $\Pi_n$-formula}\}
\end{equation*}
are \emph{not} finite extensions of Peano Arithmetic (as opposed to the case of uniform reflection). Thus we use $\prfn_\pa(\Pi_n)$ to denote an infinite set of formulas, rather than a single instance of parameter-free reflection. Beklemishev shows that the theories $\pa+\prfn_\pa(\Pi_n)$ and $\pa+\prfn_\pa(\Sigma_n)$ are incomparable if $n\geq 2$. In particular the hierarchy of parameter-free reflection principles is strict. On the other hand, there are stronger conservation results than in the uniform case: For example, the theory $\pa+\prfn_\pa(\Pi_{n+1})$ is conservative over $\pa+\prfn_\pa(\Pi_n)$ for $\Pi_n$-formulas (in case $n\geq 2$). It is straightforward to check that the proofs from \cite{beklemishev-notes-local-reflection} apply to slow parameter-free reflection as well.

\bibliographystyle{alpha}
\bibliography{Slow-Reflection_Freund}

\end{document}